\documentclass[11pt]{article}
\usepackage{geometry}                
\usepackage{graphicx}
\usepackage{amssymb}
\usepackage{amsmath}
\usepackage{amsthm}

\title{On the Spectrum of Random Anti-symmetric and Tournament Matrices}

\newtheorem{theorem}{Theorem}
\newtheorem{proposition}{Proposition}[section]

\newtheorem{lemma}[proposition]{Lemma}
\author{ Philippe Sosoe\thanks{Center for Mathematical Sciences and Applications, Harvard University. \texttt{psosoe@cmsa.harvard.edu}}\,  and Uzy Smilansky\thanks{Department of Physics of Complex Systems, Weizmann Institute, Rehovot, Israel. \texttt{uzy.smilansky@weizmann.ac.il}}}

\begin{document}

\maketitle 

\begin{abstract}
We consider a discrete, non-Hermitian random matrix model, which can be expressed as a shift of a rank-one perturbation of an anti-symmetric matrix. We show that, asymptotically almost surely, the real parts of the eigenvalues of the non-Hermitian matrix around any fixed index are interlaced with those of the anti-symmetric matrix. Along the way, we show that some tools recently developed to study the eigenvalue distributions of Hermitian matrices extend to the anti-symmetric setting.
\end{abstract}

\section{Introduction}
The purpose of this note is to analyze the small-scale properties of the spectrum of large random tournament matrices. For any positive integer $N$, a \emph{tournament} of size $N$ is an $N\times N$ matrix $D=(D_{ij})_{1\le i,j\le N}$ with entries in $\{0,1\}$ such that 
\[D_{ii}=0,\] 
and 
\[D_{ij} = 1-D_{ji}, \text{ for } i\neq j.\]
The name comes from the following interpretation of the matrix entries: $D_{ij}$ represents the outcome of the match between player $i$ and player $j$ in a tournament where every possible pair of players meets once (a ``round-robin'' tournament). We choose the matrix $D$ uniformly at random from the $2^{N(N-1)/2}$ possible choices of tournament matrices. 

The spectrum of a tournament matrix $D$ is complex, but $D$ is related to a Hermitian matrix by a simple transformation. Subtracting from each off-diagonal element of $D$ its mean $1/2$, we obtain an anti-symmetric matrix with entries in $\{\pm \frac{1}{2}\}$. Multiplying by $i$, we find that $D$ can be written in terms of a (non-Hermitian) rank one perturbation of a Hermitian matrix $M$:
\begin{equation}\label{eq: Mdef}
M = 2iD - i\left(|\mathbf{1}\rangle \langle \mathbf{1}|-I\right).
\end{equation}
Here, $\langle \mathbf{1}|=(1,\ldots, 1)$ is the row vector whose entries are all $1$, and $|\mathbf{1}\rangle = \langle \mathbf{1}|^\intercal$. We have applied an overall scaling by $2$ for convenience, so that the entries of $M/i$ lie in $\{\pm 1\}$, with variance 1, and $\mathrm{tr}\, M^2 =-N(N-1)$. 

The matrix $M$ has the form
\begin{equation}\label{eqn: realanti}
i \times (\text{real anti-symmetric matrix}).
\end{equation}
If $N$ is odd, the spectrum of the Hermitian matrix $M$ consists of the value $0$ and $N-1$ real eigenvalues symmetrically distributed in pairs about $0$. When $N$ is even, the spectrum is also symmetric, but there isn't necessarily an eigenvalue at $0$. Wigner's semicircle law implies that the spectrum is asymptotically concentrated on $[-2\sqrt{N},2\sqrt{N}]$.  

We will call matrices such as $M$, whose real part is zero, \emph{anti-symmetric Hermitian} to distinguish them from Hermitian matrices in the universality class of the commonly studied Gaussian Unitary Ensemble (GUE). The Gaussian model corresponding to matrices of the form \eqref{eqn: realanti} differs from the GUE, although its asymptotic behavior is similar in many respects. (See Section \ref{sec: antigaussian}.)

We present two results concerning the random matrices $M$ and $D$, in the large $N$ limit. The first is that the correlation functions of the matrix $M$ have sine kernel behavior in the bulk. Our second and main result relates the spectrum of $D$ to that of $M$, and more generally deals with non-Hermitian perturbations of Hermitian Wigner matrices.

Denote by 
\[\lambda_{-(N-1)/2}(M) \le \ldots \le \lambda_0(M) = 0 \le \ldots \le \lambda_{(N-1)/2}(M)\] 
the ordered eigenvalues of the matrix $M$. The eigenvalue $\lambda_0(M)$ is absent when $N$ is even.
The first result is sine kernel universality in the bulk for the matrices $M$:
\begin{theorem}\label{thm1}
Let $W$ be chosen uniformly at random from the ensemble of anti-symmetric $N\times N$ matrices with $\pm 1$ entries, and define the Hermitian matrix $M = iW$. The rescaled empirical eigenvalue density
\[\rho_M = \frac{1}{N} \sum_{j=1}^N \delta_{\frac{\lambda_j(M)}{\sqrt{N}}}\]
is symmetric about $0$. It converges almost surely in distribution to the semicircle distribution \eqref{eq: scdist}.

For any $E$ in $(0,2)$, and any $b>0$ such that $I_E=[E-b,E+b]\subset (0,2)$, the $n$-point correlation functions $p_n$ of $M$, properly rescaled and averaged over $I_E$, converge to those of the sine kernel process. That is, for any smooth $Q:\mathbf{R}^n\rightarrow \mathbf{R}$,
\begin{multline}
\iint_{I_E}Q(x_1,\ldots,x_n) p_{n}\left(E+\frac{x_1}{N\rho_{sc}(E)},\ldots, E +\frac{x_n}{N\rho_{sc}(E)}\right)\,\mathrm{d}x_1\cdots\mathrm{d}x_n \frac{\mathrm{d}E}{(\rho_{sc}(E))^n|I_E|}\\\rightarrow \iint Q(x_1,\ldots, x_n) \operatorname{det}(K(x_i-x_j))_{i,j=1}^n \,\mathrm{d}x_1\ldots\mathrm{d}x_n\mathrm{d}E,
\end{multline}
with
\[K(x-y) = \frac{\sin\pi(x-y)}{\pi(x-y)}.\]
\end{theorem}
The sine kernel behavior follows from now standard arguments developed by L. Erd\"os, B. Schlein, H.T. Yau and their collaborators. The joint distribution of the eigenvalues of an anti-symmetric Gaussian matrix can be computed explicitly and has a determinantal structure. This has been previously used to compute the asymptotic correlations between the eigenvalues in the bulk \cite{rosenzweig}. One can use the Dyson Brownian/local relaxation flow approach of \cite{localrelax}, combined with the strong local semicircle law, to show that sine-kernel universality extends from the Gaussian model to the matrices $M$.

We now turn to the second result, which describes the spectrum of $D$. Define
\[N^* =\begin{cases}
(N-1)/2 &\text{if } N \text{ is odd},\\
N/2 &\text{if } N \text{ is even}.
\end{cases}
\]
Let $\lambda_i(D)$, $-N^*\le i \le N^*$ be the eigenvalues of $D$. When $N$ is odd, there is always one real eigenvalue, which we label $\lambda_0(D)$, at distance of order $N$ from the imaginary axis. The other eigenvalues are complex and symmetrically distributed about the real axis, close to the line $\Re z = -1/2$: $\lambda_{-i}(D)=\overline{\lambda_i(D)}$. More interestingly, we prove a certain probabilistic interlacing between the eigenvalues of $M$ and the imaginary part of the eigenvalues $2D+I$. 

The Weyl interlacing theorem for Hermitian finite rank perturbations is well-known and often used in random matrix theory. Given a Hermitian matrix $A$ and a (Hermitian) positive semi-definite rank 1 perturbation $B$, the eigenvalues of $A$ and $A+B$ are interlaced. That is, there is exactly one eigenvalue of $A+B$ between any two eigenvalues of $A$:
\[ \lambda_i(A)\le \lambda_i(A+B) \le \lambda_{i+1}(A).\]

The equation \eqref{eq: Mdef} expresses $i(2D+I)$ as a rank one perturbation of the matrix $M$, but the perturbation is not Hermitian. Nevertheless, with high probability as $N$ goes to infinity, we find that the eigenvalues of $M$ are interlaced with the imaginary parts of eigenvalues of $(2D+I)$:
\begin{theorem}\label{thm2}
\begin{enumerate}
\item If $N$ is odd, the matrix $D$ has a real eigenvalue $\lambda_0(D)$ such that
\[\frac{\lambda_0(D)}{(N-1)/2}\rightarrow 1\]
almost surely as $N\rightarrow \infty$ through the positive odd integers.
\item Fix $n\in \mathbf{Z}_+$, $\epsilon>0$, and $0<\alpha<1/2$. Then, for $N=N(\epsilon)$ large enough, and each index $i\in [\alpha N^*, (1-\alpha)N^*]$, the following holds with probability greater than $1-\epsilon$: 

For each pair 
\[(\lambda_{k}(M), \lambda_{k+1}(M)), \quad  k=i,...,i+n-1,\]
of consecutive eigenvalues of $M$, the matrix $D$ has an eigenvalue whose imaginary part $\Im(\lambda_{k}(D))$ lies in the the interval 
\[(\lambda_{k}(M)/2,\lambda_{k+1}(M)/2),\] 
and, for each $k=0,...,n-1$, we have 
\[\Re(\lambda_{k}(D))=-\frac{1}{2}+O_{\epsilon'}(1/N^{1-\epsilon'}),\] 
$\epsilon'>0$ arbitrary.
\end{enumerate}
Note that the second part of the theorem applies equally to the case of odd and even $N$.
\end{theorem}

There is a large literature on low rank perturbations of random matrices. Most of this is concerned with Hermitian perturbations of Wigner matrices. See for instance \cite{beynach}, \cite{feralpeche}, \cite{furedikomlos}, \cite{knowlesyin}, \cite{knowlesyin2}, \cite{pizzoetal}, \cite{capitaineetal}, \cite{renfrew}, \cite{renfrewsoshnikov}, \cite{ttao}. The results of T. Tao in \cite{ttao} apply to non-Hermitian matrices, and those of D. Renfrew and S. O'Rourke \cite{renfrew} to non-Hermitian perturbations of Wigner matrices. These works mostly deal with the distributional properties of outlier eigenvalues created by the perturbation, either far away from the support of the limit spectral distribution, or at its edge. See \cite[Theorem 2.7]{knowlesyin} and \cite[Theorem 2.7]{bloemendaletal} for results concerning the location of bulk (non-outlier) eigenvalues for low-rank deformations of Wigner and covariance matrices. Closer to the setting of this paper, let us mention that the asymptotic correlation functions of anti-Hermitian finite-rank deformations of GUE matrices have been calculated by Y.V. Fyodorov and B.A. Khoruzhenko \cite{fyodorovkhoruzhenko}. See also \cite{fyodorovsommers} for a discussion of non-Hermitian deformations of more general Hermitian and unitary models. 

Our initial motivation was provided by the tournament matrix model, but our argument is based on a general observation about non-Hermitian perturbations of random matrices. This does not appear to have been previously noticed even in the Gaussian case. In this case, one can give a self-contained proof which does not rely on recent universality results for the eigenvectors. We remark on this in Section \ref{sec: gueext}.

We end this introduction by noting that it would be interesting to describe the spectral statistics of \emph{regular} tournaments, that is, tournament matrices satisfying the row sum constraint $\sum_{j}D_{ij}=(N-1)/2$. Indeed, the results presented here were obtained as we were investigating the statistical properties of the spectra of this restricted ensemble of tournament matrices.

\subsection{Organization of the paper}
In Section \ref{sec: antigaussian}, we describe some properties of the Gaussian model for the anti-symmetric Hermitian class. We also summarize the computation of the correlation functions in the bulk. We present the equations for an appropriate Dyson dynamics on the eigenvalues and eigenvectors in Sections \ref{sec: antidbm} and \ref{sec: vantidbm}.

The proofs of Theorem \ref{thm1} and Theorem \ref{thm2} (2), rely on analogs for anti-symmetric Hermitian matrices, of results from recent work of P. Bourgade, L. Erd\"os, B. Schlein, H.T. Yau, and J. Yin, \cite{bourgadeyau,ey,localrelax}, concerning symmetric and Hermitian matrices.  We will present the modifications necessary in order to adapt the proofs of these results to anti-symmetric matrices. Replacing the dynamics in \cite{localrelax} by the anti-symmetric Dyson eigenvalue flow, one immediately obtains sine kernel universality in the bulk by the method in that paper; this is Theorem \ref{thm1}. We summarize the argument in Section \ref{sec: sine}.

The proof of Theorem \ref{thm2} (1) appears in Section \ref{sec: interlacing}. The proof of Part (2) of the theorem in  Section \ref{sec: inter2} is more intricate, and requires the results from \cite{bourgadeyau,bloemendaletal,ey}, suitably adapted to our case. The proofs of both parts of Theorem \ref{thm2} are based on the perturbation equation \eqref{eqn: Fsdef} for the eigenvalues of $2D+I$. The precise information from \cite{bourgadeyau, ey}
allows us to analyze that equation essentially at the level of single eigenvalue spacings, and locate the eigenvalues of $2D+I$. 

A central estimate is contained in Lemma \ref{lma: separation}, whose proof appears in Section \ref{sec: sep-lma}. The argument uses an anti-concentration result on the eigenvectors of the matrix $M$, which we obtain in Section \ref{bysection} by adapting a result of Bourgade and Yau \cite{bourgadeyau} to the anti-symmetric setting. Some relevant computations are relegated to Appendix \ref{sec: by}. 

Finally, in Section \ref{sec: gueext}, we remark that the argument for Theorem \ref{thm2}, (2) applies more generally to certain perturbations of Hermitian Wigner matrices in the GUE class.

\subsection{Acknowledgments}
We thank Michael Aizenman for his continued interest in this work. P.S. would like to thank Paul Bourgade for his explanations concerning various aspects of the paper \cite{bourgadeyau}, the Weizmann Institute for its hospitality during a visit where this work was begun, and the Center for Mathematical Sciences and Applications at Harvard University for financial support.

\newpage

\section{The Gaussian anti-symmetric model}
\label{sec: antigaussian}
In this section we introduce a Gaussian anti-symmetric matrix model \eqref{eqn: antiguemodel}, and compute its correlation functions as $N\rightarrow \infty$ (Section \ref{sec: correlation}). In Sections \ref{sec: antidbm} and \ref{sec: vantidbm}, we present the anti-symmetric analogs of the Dyson Brownian motion equations. These are used in the proofs of the results from \cite{bourgadeyau,localrelax,bulkuniv,ey} we appeal to.

For the remainder of the paper, $N$ will denote an odd positive integer. The proofs in the case of even $N$ are for the most part identical, and sometimes simpler.
\subsection{The matrix model}
Consider a real $N\times N$ anti-symmetric matrix $K$, such that the entries $g_{ij}$, $1 \le i<j\le N $ above the diagonal have distribution $\frac{1}{\sqrt{2\pi}}e^{-x^2/2}$, and the diagonal entries are zero:
\begin{align} \label{eqn: antiguemodel}
g_{ij} &=-g_{ji}, \quad 1\le i<j\le N\\
g_{ii}&=0.  \quad i=1,\ldots, N \nonumber
\end{align}
We begin with a few elementary consequences of the general form of $K$. By the anti-symmetry condition, $K$ has $N$ imaginary eigenvalues $i\nu_{-(N-1)/2}, \ldots, \nu_0, \ldots, i\nu_{(N-1)/2}$, with corresponding unit eigenvectors $v_j=v_j(K)$, $j=-(N-1)/2, \ldots, (N-1)/2$. Since $K$ is real, we have
\[K \overline{v_j}=\overline{K v_j} = -i\nu_j \overline{v_j},\]
implying that the spectrum is symmetric about the real axis:
\[\nu_{-j} = -\nu_j,\]
and the eigenvectors satisfy:
\begin{align*}
v_{-j}&=\overline{v_{j}}, \  j=1, \ldots, (N-1)/2\\
v_j &\in \mathbf{R}^{N},\  j=0.
\end{align*}
The matrix $H=-iK$ is Hermitian, with eigenvalues $\lambda_j=\nu_j$ and the same eigenvectors as $H$. By orthogonality of the eigenvectors with respect to the Hermitian inner product, we have, for $1\le i,j \le N$:
\[ v_i\cdot \overline{v_{-j}}= 0 = v_i \cdot v_j.\]
Here, we have denoted by $\cdot$ the real inner product on $\mathbf{R}^N$. Expanding and taking real and imaginary parts, we have
\begin{align*}
\Re v_i\cdot \Re v_j &= \Im v_i\cdot \Im v_j = 0\\
\Im v_i \cdot \Re v_j &=  \Im v_j \cdot \Re v_i  = 0
\end{align*}
We thus conclude:
\begin{enumerate}
\item The $N$ vectors $\Re v_1, \ldots, \Re v_{(N-1)/2}, v_0, \Im v_1, \ldots \Im v_{(N-1)/2}$ form an orthogonal basis of $\mathbf{R}^N$ with the usual real Euclidean inner product.
\item $|v_0|^2=\frac{1}{2}$, and $|\Re v_j|^2 = |\Im v_j |^2 = \frac{1}{2}$ for $j=1, \ldots, (N-1)/2$.
\end{enumerate}

If $O$ is a real orthogonal matrix ($OO^t=I$), then the $(i,j)$ entry of the matrix
$OK(O)^{-1}= OKO^t$ is given by
\[\sum_{k,l=1}^N O_{il}O_{jk}g_{kl} = -\sum_{k,l=1}^N O_{jk}O_{il}g_{lk},\]
so that the conjugated matrix has the same distribution as $K$. Given a Borel subset $A \subset O(N)$, the orthogonal group, and any orthogonal matrix $O$, we have
\begin{align*}
 &\mathbf{P}\left( \sqrt{2}(\Re v_j(H),\Im v_j(H)) \in A\right)\\
= \ & \mathbf{P}\left( \sqrt{2}( \Re v_j( O^{-1}HO), \Im v_j(O^{-1}HO)) \in OA \right)\\
= \  &\mathbf{P}\left( \sqrt{2}(\Re v_j(H),\Im v_j(H)) \in OA \right).
\end{align*}
In the second step we have used the distributional invariance of the matrix under conjugation. So the anti-symmetric Gaussian distribution on the entries of $K$ induces the Haar distribution on
\[\sqrt{2}(v_0,\Re v_1,\ldots,\Re v_{(N-1)/2},\Im v_1,\ldots,\Im v_{(N-1)/2})\in O(N).\]

\subsection{Correlation functions in the bulk}\label{sec: correlation}
Eigenvalue correlation functions for the anti-symmetric Gaussian ensemble, as well as their scaling limits, can be computed explicitly. This was done long ago by M. L. Mehta and N. Rosenzweig \cite{rosenzweig}. We summarize the result of their computation.

The matrix $H=-iK$ is Hermitian, and the probability density function of its eigenvalues is given by (see \cite[Eq. (10)]{rosenzweig})
\begin{equation}
Z_N^{-1}\prod_{1\le i<j \le (N-1)/2} |\lambda_i^2-\lambda_j^2|^2\prod_{i=1}^{(N-1)/2}\lambda_i^2 e^{-\frac{\lambda_i^2}{2}}.\label{eq: pdf}
\end{equation}

The Vandermonde determinant in \eqref{eq: pdf} is equals 
\begin{align}
&\prod_{1\le i \le (N-1)/2}  \lambda_i^2 \prod_{1\le i< j \le (N-1)/2} |\lambda_i^2-\lambda_j^2|^2 \nonumber \\
=& \ \mathrm{det} \left(\begin{array}{cccc} \label{eqn: det}
\lambda_1 & \lambda_1^3 & \cdots & \lambda_1^{N-2}\\
& & \ddots  & \\
\lambda_{(N-1)/2} & \lambda_{(N-1)/2}^3 & \cdots & \lambda_{(N-1)/2}^{N-2}.
\end{array}\right)^2.
\end{align}
Let $P_k(x)$ denote the $k$-th order (monic) Hermite polynomial
\[P_k(x) = (-1)^k e^{x^2/2}\left(\frac{\mathrm{d}}{\mathrm{d}x}\right)^ke^{-x^2/2}.\]
By multilinearity, the determinant in \eqref{eqn: det} is expressed in terms of these polynomials:
\[\left(\mathrm{det} \left(p_{2k-1}(\lambda_i)\right)_{1\le  i,k \le (N-1)/2}\right)^2 = \mathrm{det}(K_N(\lambda_i,\lambda_j))_{1\le i,j\le N},\]
with $K_N(x,y)=2\sum_{k=0}^{(N-1)/2-1}P_{2k-1}(x)P_{2k-1}(y)$. 

Replacing the polynomials $P_k(x)$ with the normalized Hermite functions
\[\psi_k(x) = \frac{1}{(\sqrt{2\pi}k!)^{1/2}}P_k(x)e^{-x^2/4},\]
we find that the $k$-point correlation function for the positive eigenvalues has the determinantal expression
\[p_{k}(\lambda_1,...,\lambda_k) =  \frac{(N-k)!}{N!}\cdot \mathrm{det}(\tilde{K}_N(\lambda_i,\lambda_j))_{1\le i,j\le N}, \quad 1\le k\le (N-1)/2\]
with
\[\tilde{K}_N(x,y) = 2\sum_{k=0}^{(N-1)/2-1}\psi_{2k-1}(x)\psi_{2k-1}(y),\]
and
\[\int_{\lambda_1,\ldots,\lambda_k\ge 0}p_k(\lambda_1,...,\lambda_k)\,\mathrm{d}\lambda_1\ldots\lambda_k=1.\]

The argument used to prove the Christoffel-Darboux formula \cite[Section 3.2.1]{guionnetzeitouni} gives the relation
\begin{align}
\tilde{K}_N(x,y) &= \left(\frac{N-1}{2}\right)^{1/2}\frac{ \psi_{N-1}(x)\psi_{N-2}(y)-\psi_{N-1}(x)\psi_{N-2}(y) }{ x-y } \label{KN-kernel}\\
& \quad +\left(\frac{N-1}{2}\right)^{1/2} \frac{ \psi_{N-1}(x)\psi_{N-2}(y)+\psi_{N-1}(x)\psi_{N-2}(y) }{ x+y }. \nonumber
\end{align}
Recall the asymptotics for the Hermite functions:
\begin{align}
\psi_{2k}(x) = \frac{(-1)^k}{N^{1/4}\sqrt{\pi}}\cos(\sqrt{N}x)+o(N^{-1/4}) \label{PRasymp1}\\
\psi_{2k+1}(x) = \frac{(-1)^k}{N^{1/4}\sqrt{\pi}}\sin(\sqrt{N}x)+o(N^{-1/4}), \label{PRasymp2}
\end{align}
uniformly for $x$ in a compact set.
Combining \eqref{KN-kernel} and \eqref{PRasymp1}, \eqref{PRasymp2}, we have the approximation
\[\tilde{K}_N(x,y)\sim \frac{\sqrt{2(N-1)}}{\pi}\left(\frac{\sin(x-y)}{x-y}+\frac{\sin(x+y)}{x+y}\right).\]
This formula gives the scaling limit of the correlation functions at 0. At energies $x$, $y$ of order $\sqrt{N}E$ with $0<E<2$, the asymptotics \eqref{PRasymp1}, \eqref{PRasymp2} we find the sine kernel behavior
\[\tilde{K}_N(x,y)\sim \frac{\sqrt{2(N-1)}}{\pi}\frac{\sin(x-y)}{x-y}.\]
as in the GUE case.
\subsection{Anti-symmetric Dyson Brownian Motion} \label{sec: antidbm}
In this section, we define the \emph{anti-symmetric Dyson Brownian motion}. 

Given an $N\times N$ anti-symmetric Hermitian matrix $M_0$, we say that the matrix-valued process $t\mapsto H(t)$, $t\ge 0$ is an \emph{anti-symmetric Brownian motion} started at $M_0$ if
\begin{equation}\label{eqn: matrixDBM}
H(t) = M_0 + \frac{i}{\sqrt{N}}G(t),
\end{equation}
where 
\begin{gather}
(G_{i,j}(t))_{i>j, t\ge 0} \text{ is an } \mathbf{R}^{N(N-1)/2}\text{-valued Brownian motion} \label{eqn: asBM1}\\
G_{ij}(t)=-G_{ji}(t) \text{ for all } 1\le i,j \le N \text{ and } t \ge 0.\label{eqn: asBM2}
\end{gather}
We will also consider the \emph{anti-symmetric Ornstein-Uhlenbeck process}, for which the $N(N-1)/2$-dimensional Brownian motion in condition \eqref{eqn: asBM1} is replaced by an Ornstein-Uhlenbeck process $(G^{OU}(t)_{ij})$, $1\le j\le i\le N$.

We define a set of stochastic differential equations associated with \eqref{eqn: matrixDBM}. Let 
\[\vec{\lambda}(0) = (\lambda_1(M_0), \ldots, \lambda_{(N-1)/2}(M_0))\]
be a vector of positive values. The \emph{anti-symmetric Dyson eigenvalue flow driven by Brownian motion} started at $\vec{\lambda}(0)$ is the solution of the system
\begin{equation}\label{eqn: asdbm}
\mathrm{d}\lambda_j = \frac{\mathrm{d}G_{2j-1,2j}}{\sqrt{N}}+\left(\frac{1}{N}\sum_{l\neq j} \frac{1}{\lambda_j-\lambda_l}+\frac{1}{N}\sum_{l\neq j} \frac{1}{\lambda_j+\lambda_l}+\frac{1}{N}\frac{1}{\lambda_j}\right)\mathrm{d}t,
\end{equation}
for
\[1\le j \le (N-1)/2.\]
Here $G$ is as in \eqref{eqn: asBM1}, \eqref{eqn: asBM2}. We will also have use for the \emph{anti-symmetric Dyson eigenvalue flow driven by the Ornstein-Uhlenbeck process}:
\begin{equation}\label{eqn: asdbm-ornstein}
\mathrm{d}\lambda_j = \frac{\mathrm{d}G_{2j-1,2j}}{\sqrt{N}}+\left(-\frac{1}{2N}\lambda_j+\frac{1}{N}\sum_{l\neq j} \frac{1}{\lambda_j-\lambda_l}+\frac{1}{N}\sum_{l\neq j} \frac{1}{\lambda_j+\lambda_l}+\frac{1}{N}\frac{1}{\lambda_j}\right)\mathrm{d}t.
\end{equation}

The essential point is that for each $t\ge 0$, the distribution of the positive eigenvalues of the matrix evolution \eqref{eqn: matrixDBM} coincides with that of the solution of \eqref{eqn: asdbm}. A similar statement holds for the solutions of \eqref{eqn: asdbm-ornstein}, provided we replace the Brownian motion in \eqref{eqn: asBM1} by a matrix-valued Ornstein-Uhlenbeck process.

\subsection{Anti-symmetric eigenvector flow}
\label{sec: vantidbm}
The main tool to derive Lemma \ref{eqn: quelemma} in Section \ref{bysection} will be the \emph{anti-symmetric} Dyson Brownian vector flow. The symmetric and Hermitian versions of this  system were analyzed P. Bourgade and H.T. Yau \cite{bourgadeyau}, and we will use their methods.

The anti-symmetric eigenvector flow is the $(N+1)/2$-dimensional system of stochastic differential equations
\begin{align}
\mathrm{d}v_0 &= \frac{i}{\sqrt{2}} \sum_{l\neq 0}\frac{\mathrm{d}G_{0,2l-1}-i\mathrm{d}G_{0,2l}}{\lambda_l} v_l +\frac{1}{\sqrt{2}}\sum_{l\neq 0}\frac{\mathrm{d}G_{0,2l-1}-i\mathrm{d}G_{0,2l}}{\lambda_l}\overline{v_l} \label{eqn: ev0DBM}  \\
&\quad -\frac{1}{\sqrt{2}} \sum_{l\neq 0} \frac{\mathrm{d}t}{\lambda_l^2}, \nonumber \\
\mathrm{d}v_k &= \frac{i}{2}\sum_{k\neq l} \frac{\mathrm{d}G_{2k-1,2l-1}+i\mathrm{d}G_{2k,2l-1}-i\mathrm{d}G_{2k-1,2l}+\mathrm{d}G_{2k,2l}}{\lambda_k-\lambda_l} \label{eqn: evDBM} \\
&\ +\frac{i}{2}\sum_{k\neq l} \frac{\mathrm{d}G_{2k-1,2l-1}+i\mathrm{d}G_{2k,2l-1}+i\mathrm{d}G_{2k-1,2l}-\mathrm{d}G_{2k,2l}}{\lambda_k+\lambda_l} \nonumber\\
&\ +\frac{i}{\sqrt{2}}\frac{\mathrm{d}G_{2k-1,0}+i\mathrm{d}G_{2k,0}}{\lambda_k} \nonumber \\
&\ -\frac{i}{2}\sum_{k\neq l} \left(\frac{\mathrm{d}t}{(\lambda_k-\lambda_l)^2}v_l+\frac{\mathrm{d}t}{(\lambda_k+\lambda_l)^2}\overline{v_l}+\frac{\mathrm{d}t}{\lambda_k^2}v_0 \right). \nonumber
\end{align}
Here, $G$ is an $N$-dimensional anti-symmetric matrix Brownian motion as in \eqref{eqn: asBM1}, \eqref{eqn: asBM2}. The system is expressed in complex coordinates $v_l=\Re v_l +i \Im v_l$. Coupled with the anti-symmetric eigenvalue flow \eqref{eqn: asdbm},  the equations \eqref{eqn: ev0DBM}, \eqref{eqn: evDBM} form the anti-symmetric Dyson Brownian motion.

The significance of these equations derives from the following: if we let 
\begin{equation}\label{eqn: sol}
\left(\vec{\lambda}(t),(v_j(t))_{0\le j\le \frac{N-1}{2}})\right),
\end{equation}
be the solution of \eqref{eqn: asdbm}, \eqref{eqn: ev0DBM}, \eqref{eqn: evDBM} with initial data $\vec{\lambda}(0)=\vec{\lambda}(M_0)$ and 
\[(v_0(0), \ldots,v_{\frac{N-1}{2}}(0))=\left(v_0(M_0),v_1(M_0),\ldots, v_{\frac{N-1}{2}}(M_0)\right),\]
then the distribution of the vector \eqref{eqn: sol} coincides with that of the (non-negative) eigenvalues of anti-symmetric Dyson Brownian motion \eqref{eqn: matrixDBM} and their associated eigenvectors.

\section{Sine Kernel Universality for anti-symmetric matrices}\label{sec: sine}
The object of this section is the derivation of Theorem \ref{thm1}. The first input in the proof is the \emph{local semicircle law} (see \cite{sc1, sc2}). The $N\times N$ matrix $M$ is a Hermitian matrix with independent, identically distributed entries, except for the symmetry constraint. The proof of the strong local semicircle given in \cite{sc1} applies to the matrix $M$. Concerning this last point, we note that in the recent literature, ``Hermitian matrices'' are generally assumed to satisfy additional non-degeneracy conditions on the real and imaginary parts. Such conditions ensure that the eigenvalue and eigenvector distributions on small scales asymptotically coincide with those of the GUE. The results in \cite{sc1} apply to random matrices which are Hermitian in the usual sense, see \cite[Section 2]{sc1}.

The following version of the semicircle law follows directly from \cite[Theorem 2.3]{sc1}.
\begin{theorem}[Local semicircle law] \label{eqn: sclaw} For an $N\times N$ matrix $M$ as in \eqref{eq: Mdef} and $z=E+i\eta$, $\eta>0$, define the resolvent matrix
\[R(z) = \left(\frac{1}{\sqrt{N}}M-z\right)^{-1},\]
and its normalized trace
\[m_N(z) = \frac{1}{N}\operatorname{tr}R(z),\]
as well as the Stieltjes transform of the semicircle distribution
\begin{equation}\label{eqn: mscdef}
m_{sc}(z) =\int_{\mathbf{R}} \frac{1}{x-z}\rho_{sc}(x)\,\mathrm{d}x,
\end{equation}
with
\begin{equation}
\label{eq: scdist}
\rho_{sc}(x) = \frac{1}{2\pi} \mathbf{1}_{[-2,2]}(x)\sqrt{4-x^2}.
\end{equation}
Then, for $\epsilon,c>0$, and $k>0$, we have
\begin{equation}
\label{eq: stronglocalsc}
|m_N(z)-m_{sc}(z)|\le \frac{1}{N^{1-\epsilon}\eta},
\end{equation}
uniformly in the energy $E=\Re z\in [-10,10]$ and $\eta \in [ N^{-1+c},1]$ with probability greater than $1-C_kN^{-k}$. The probability refers to the uniform measure over all anti-symmetric matrices with entries in $\{\pm 1\}$.
Moreover, for each $c,\delta>0$, and $k>0$ we have the entry-wise bound:
\begin{equation}
\sup_{i,j} |R_{ij}(z)-\delta_{ij}m_{sc}(z)|\le \frac{1}{\sqrt{N^{1-\epsilon}\eta}}
\end{equation}
uniformly for $\delta < E<2-\delta$ and $\eta \in [ N^{-1+c},1]$, with probability greater than $1-CN^{-k}$.
\end{theorem}

The local semicircle law, Theorem \ref{eqn: sclaw}, is the input to the local relaxation flow approach developed in \cite{localrelax} by Erd\"os, Schlein, Yau and Yin. We now explain their argument, adapted to our context. 

Consider a matrix of the form
\begin{equation}\label{eqn: flow}
T(t)=e^{-t/2}M_0+(1-e^{-t})^{1/2}iG,
\end{equation}
where $M_0$ is an anti-symmetric Wigner matrix and $G$ is an independent Gaussian matrix as in \eqref{eqn: antiguemodel}. For each $t\ge 0$, the distribution of each entry $T_{ij}(t)$ coincides with that of an Ornstein-Uhlenbeck flow with initial distribution $(M_0)_{ij}$. If $G$ is normalized so that
\[G_{ij} = \frac{1}{\sqrt{N}}g_{ij}, \ g_{ij} \sim N(0,1),\]
for $i>j$, the distribution at time $t$ of the $(N-1)/2$ positive eigenvalues of $T(t)$ coincides with that of an anti-symmetric Dyson Brownian motion, given by \eqref{eqn: asdbm-ornstein}, with the eigenvalues of $T(0)=M_0$ as initial condition
\[\vec{\lambda}(0) = (\lambda_1(M_0), \ldots, \lambda_{(N-1)/2}(M_0)).\]

The equilibrium measure for the dynamics on $\{\lambda_i(t)\}_{1\le i\le (N-1)/2}$ is given by \eqref{eq: pdf}. This measure is of the form required by \emph{Assumption I'} in \cite{localrelax}. That is, the asymptotic distribution of \eqref{eqn: flow} for $t\rightarrow \infty$ can be written as
\begin{equation}
\mu_N(\mathrm{d}\mathbf{x})= \frac{e^{-\beta \mathcal{H}_N(\mathbf{x})}}{Z_\beta},
\end{equation}
\[\mathcal{H}_N(\mathbf{x})= \sum_{j=1}^N U(x_j) -\frac{1}{N}\sum_{i<j} \log|x_i-x_j|-\frac{1}{N}\sum_{i<j}\log|x_i+x_j|-\frac{c_N}{N}\sum_{j}\log|x_j|.\]
The density \eqref{eq: pdf} takes this form with $\beta=2$ and $U(x_j)=x_j^2$.

By \cite[Theorem 2.1]{localrelax} and the local semicircle law \eqref{eq: stronglocalsc}, we have the following
\begin{theorem}
Let $E\in (0,2)$ with $b>0$ such that $[E-b,E+b]\subset (0,2)$. Denote by $p_{t,n,N}$ the $n$-point correlation functions of the matrix model $T(t)/N^{1/2}$ \eqref{eqn: flow} of size $N$ and by $p_{n,N,G}$ the $n$-point correlation function of the Gaussian anti-symmetric model. We have, for any smooth, compactly-supported real-valued function $Q$, and any $\epsilon>0$:
\begin{multline} \label{DBMflow} \lim_{N\rightarrow \infty} \sup_{t\ge N^{-1+\epsilon}}\int_{E-b}^{E+b}\frac{\mathrm{d}E'}{2b}\int_{\mathbf{R}^d} Q(x_1,\ldots,x_n) \\
\times \frac{1}{\rho_{sc}(E)}(p_{t,n,M}-p_{n,N,G})\left(E'+\frac{x_1}{N\rho_{sc}(E)}, \ldots, E'+\frac{x_n}{N\rho_{sc}(E)}\right)\,\mathrm{d}x_1\ldots \mathrm{d}x_n =0.
\end{multline}
\end{theorem}

Given the local semicircle law \eqref{eq: stronglocalsc} and the short-time universality result \eqref{DBMflow}, Theorem \ref{thm1} follows by the Green function comparison method of Erd\"os, Yau and Yin, introduced in \cite[Theorem 2.3]{bulkuniv}. Their method uses the local semicircle law and entry-wise substitution to conclude that when the first four moments of the entries of two matrices in the same symmetry class coincide approximately, so do the local eigenvalue statistics. Choosing a random matrix $M_0$ such that the first four moments of the corresponding matrix $T(t)$ coincide with $M$ in \eqref{eq: Mdef} asymptotically as in \cite[Lemma 3.4]{bernoulliuniv}, we obtain Theorem \ref{thm1}.

This procedure is by now standard. Since its application in our case presents no particular difficulty, we do not reproduce it here. See \cite[Section 3.2]{erdoescdm} or  \cite[Section 8]{sc1}, for detailed expositions.

\section{Proof of Theorem \ref{thm2}, Part (1)}
\label{sec: interlacing}

Before proceeding with the proof of Theorem \ref{thm2}, we introduce some notation. 

For a sequence $A_N$, $N\ge 1$, and $b\in \mathbf{R}$ we write 
\[A_N \lesssim N^b\]
if, for each $\epsilon, \epsilon'>0$, there is a constant $C(\epsilon, \epsilon')$ such that for all $N$ sufficiently large,
\[A_N\le C(\epsilon,\epsilon')N^{b+\epsilon},\] with probability at least $1-\epsilon'$. We also write $A_N \asymp N^b$ if for each $\epsilon,\epsilon'>0$, there is some $C(\epsilon,\epsilon')>0$ such that for all $N$ sufficiently large 
\[(1/C(\epsilon))\cdot N^{b-\epsilon} \le A_N \le C(\epsilon)N^{b+\epsilon}\]
with probability at least $1-\epsilon'$.

\begin{proof}
The matrix 
\[2D+I = -iM+|\mathbf{1}\rangle\langle \mathbf{1}|\] 
is a rank-one perturbation of the (anti-Hermitian) matrix $-iM$. Let $v_{j}(M)$, $j=-(N-1)/2, \ldots, (N-1)/2$ be the eigenvectors of $M$ corresponding to the eigenvalues $\lambda_j(M)$. Recall that $v_{-j}(M)=\overline{v_j(M)}$. Define the function $F:\mathbf{C}\rightarrow \mathbf{C}$ by
\begin{align}
F(s)&=\sum_{j=-(N-1)/2}^{(N-1)/2}\frac{|\langle \mathbf{1}, v_j(M) \rangle|^2}{s-i\lambda_j(M)} \label{eqn: Fsdef} \\
&=\frac{|\langle \mathbf{1}, v_0(M)\rangle|^2}{s}+2s\sum_{j=1}^{(N-1)/2}\frac{|\langle \mathbf{1}, v_j(M) \rangle|^2}{s^2+\lambda_j(M)^2}. \label{eqn: positive}
\end{align}
The eigenvalues of $2D+I$ are the solutions of the equation
\begin{equation}
F(s)=1.\label{eqn: Fseq}
\end{equation}
This equation is well-known. For a derivation based on a determinant identity, which applies also to higher rank perturbations, see \cite[Lemma 2.1]{ttao}.

The expression in \eqref{eqn: positive} is positive for positive, real $s$. It has a pole at $s=0$ and decreases monotonically to zero as $s\rightarrow \infty$. $F(s)$ is also an odd function of $s$, viewed as a real parameter. It follows that the equation \eqref{eqn: Fseq} has a single real solution $2\lambda_0(D)+1$, corresponding to an eigenvalue $\lambda_0(D)$ of $D$. To establish point (1) in Theorem \ref{thm2}, it remains to prove the statement regarding the location of $2\lambda_0(D)+1$. 

By \cite[Theorem 2.2]{renfrew} (with the parameter $\rho$ there equal to $-1$), for any $\delta>0$, all eigenvalues of the real anti-symmetric matrix $(1/i)M$ are almost surely contained in the the region
\begin{equation}
\label{eqn: regionE}
\sqrt{N}\mathcal{E}_{-1,\delta} =\{\mathrm{dist}([-2i\sqrt{N},2i\sqrt{N}],s)\le \delta \sqrt{N}\}
\end{equation}
for sufficiently large $N$. In particular, $|\lambda_j(M)|\le (2+\delta)\sqrt{N}$ for all $j$, so we can rewrite $F(s)$ as
\begin{align*}
\ &\frac{1}{s}\sum_{j=-(N-1)/2}^{(N-1)/2}|\langle \mathbf{1}, v_j(M)\rangle|^2 + \sum_{j=-(N-1)/2}^{(N-1)/2}i\lambda_j(M)\frac{|\langle \mathbf{1}, v_j(M)\rangle|^2}{(s-\lambda_j(M))s}\\
=\ & \frac{N}{s}+\sum_{j=-(N-1)/2}^{(N-1)/2}i\lambda_j(M)\frac{|\langle \mathbf{1}, v_j(M)\rangle|^2}{(s-\lambda_j(M))s}.
\end{align*}
If $s$ is in the complement of the region $\mathcal{E}_{-1,3\delta}$, this last expression is
\[N\left(\frac{1}{s} + O(1/\sqrt{N})\right).\]
By Rouch\'e's theorem, $N/s-1$ and $F(s)-1$ have the same number of zeros in the circle $\{|z-N|\le \sqrt{N}\}$ for large $N$, and so the unique real solution of $F(s)=1$ is equal to  $N(1+o(1))$. The first part of Theorem \ref{thm2} is now proved.
\end{proof}

\section{Proof of Theorem \ref{thm2}, Part (2)}\label{sec: inter2}
We turn to part (2) of the theorem. We use the notation introduced in the previous section throughout. We will exclusively refer to the eigenvalues of the anti-symmetric matrix $M$, and will occasionally write $\lambda_j$ for $\lambda_j(M)$ for simplicity.

On the imaginary axis, the function $F(s)$ is purely imaginary, with monotone imaginary part between $\lambda_j(M)$ and $\lambda_{j+1}(M)$, whenever 
\[|\langle v_j(M),\mathbf{1}\rangle|\neq 0 \text{ and } \lambda_j(M)\neq \lambda_{j+1}(M).\]
$F(s)$ has poles at $i\lambda_j(M)$ for each $j=-(N-1)/2,\ldots, (N-1)/2$.
Thus there is a unique
\[\mu_j(M)\in (\lambda_j(M),\lambda_{j+1}(M))\] 
such that
\[F(i\mu_j(M))=0.\]

Around the zero $i\mu_j$, we expand the function $F$ into a Taylor series. For $z\in\mathbf{C}$ with $|z|$ sufficiently small, 
\[F(i\mu_j(M) + z) = \sum_{m\ge 1} \frac{F^{(m)}(i\mu_j(M))}{m!}z^m\]
We estimate $F^{(m)}(i\mu_j)$, $m=0,1$, to find a solution of $F(z)=1$ near $z=0$ by truncating the Taylor series after the first term.

The estimates on $F$ will be obtained directly from the definition \eqref{eqn: Fsdef}. This requires bounds for
\begin{enumerate}
\item The distances $|\mu_k(M)-\lambda_j(M)|$ between a fixed zero and the eigenvalues.
\item The overlaps $|\langle \mathbf{1}, v_j(M)\rangle|^2$ between the perturbation and the eigenvectors.
\end{enumerate}
The relevant estimates appear in Sections \ref{sec: gaps} and \ref{sec: innerprod}, respectively. Using these bounds, we prove
\begin{proposition}\label{prop: exist}
For each $\epsilon>0$, and $i\in [\alpha (N-1)/2, (1-\alpha)(N-1)/2]$, there exists $s_j\in\mathbf{C}$ with $\lambda_j(M) \le \Im s_j \le \lambda_{j+1}(M)$ and  $\Re s_j \le N^{-1+\epsilon}$
such that
\[F(s_j)=1.\]
with probability $1-o(1)$.
\end{proposition}
Theorem \ref{thm2}, (2) follows directly from this.

\subsection{Estimate on gaps}\label{sec: gaps}
For the separation between the zeros and the surrounding eigenvalues, we have the following key estimate:
\begin{lemma}[Separation of $\mu_j$ from the eigenvalues.]
\label{lma: separation}
Let $i\in [\alpha (N-1)/2,(1-\alpha)(N-1)/2]$. Then:
\begin{equation}\label{eqn: separationeq}
\min\{ \mu_j(M)-\lambda_j(M), \lambda_{j+1}(M)-\mu_{j}(M)\} \gtrsim N^{-1/2},
\end{equation}
for $j=i,\ldots, i+n-1$.
\end{lemma}
The proof of Lemma \ref{lma: separation} appears in Section \ref{sec: sep-lma}. An essential ingredient in the proof is the following:
\begin{proposition}[Eigenvalue spacing]
\label{spacingprop}
Let $i\in [\alpha (N-1)/2,(1-\alpha)(N-1)/2]$. We have
\begin{equation}\label{eqn: spacing}
|\lambda_{i+l}(M)-\lambda_{i+l+1}(M)|\asymp N^{-1/2},
\end{equation}
for $l=0,\ldots, n-1$.
\end{proposition}
To prove this, we will use another powerful universality result, due to L. Erd\"os and H.T. Yau \cite{ey}, concerning universality of statistics of gaps between eigenvalues (see Section \ref{sec: repulse}).

\subsection{Estimates for the inner products}\label{sec: innerprod}
\noindent \emph{Upper bound: delocalization of eigenvectors.}
In \cite{bloemendaletal}, the authors derive a general isotropic eigenvector delocalization result as a consequence of their main theorem. Let $H$ be an $N\times N$ Hermitian random matrix with independent entries, such that all entries have comparable variance:
\[\frac{1}{CN} \le \mathbf{E} |H_{ij}|^2 \le \frac{C}{N},\]
for some constant $C$ independent of $N$ and $i,j$. Assume also that
\[\mathbf{E}|\sqrt{N}H_{ij}|^p \le C_p.\]
Then \cite[Theorem 2.16]{bloemendaletal}, we have
\begin{equation}
\label{eqn: deloc-statement}
\max_{-(N-1)/2 \le j \le (N-1)/2 } |\langle v_j(M), \mathbf{v}\rangle|^2 \lesssim N^{-1},
\end{equation}
for any unit vector $\mathbf{v}\in\mathbf{C}^{(N-1)/2}$. Such a statement, and the isotropic local semicircle law from which it is derived, first appeared in \cite{knowlesyin}. The hypotheses in that work are only slightly too restrictive to accommodate our purely imaginary Hermitian matrix $M$.
 
Applying \eqref{eqn: deloc-statement} to the unit vector $\mathbf{v}=\langle \mathbf{1}|/\sqrt{N}$, we have
\begin{equation}\label{eqn: isotropic}
\max_{j}|\langle \mathbf{1}, v_j(M)\rangle| \lesssim 1.
\end{equation}

\noindent{\emph{Lower bound: Bourgade-Yau estimate.}}
In Section \ref{bysection} we prove the following by adapting methods of P. Bourgade and H.T. Yau from \cite{bourgadeyau}.
\begin{proposition}\label{prop: lowerbd}
For any index $j\in 1,\ldots, (N-1)/2$, any unit $\mathbf{u}$, and sequence $(a_N)_{N\ge 0}$ with $a_N \downarrow 0$, and $N$ sufficiently large, 
\begin{equation}
\label{eqn: problowerbd}
\mathbf{P}(\sqrt{N}|\langle \mathbf{u}, v_j(M) \rangle| \ge a_N)=1-o(1).
\end{equation}
\end{proposition}
In particular, a union bound implies that 
\begin{equation}
\label{eqn: bdbelow}
\max_{k\in \{i,\ldots, i+n\}}|\langle v_{k}(M), \mathbf{1}\rangle|\gtrsim 1.
\end{equation}

\subsection{Proof of Proposition \ref{prop: exist}}
Assuming Lemma \ref{lma: separation}, Lemma \ref{spacingprop} and Proposition \ref{prop: lowerbd} for now, we proceed with the proof of Proposition \ref{prop: exist}.

\begin{proof}[Proof of Proposition \ref{prop: exist}]
Let $i\in [\alpha (N-1)/2, (1-\alpha)(N-1)/2]$ be as in the statement of the theorem, and $k\in [i, \ldots, i+n-1]\cap \mathbf{Z}$. We will find a solution of the equation $F(s)=1$ with imaginary part $\Im s$ between $\lambda_k(M)$ and $\lambda_{k+1}(M)$, with probability $1-o(1)$. The theorem then follows by an application of the union bound, since $n$ is fixed.

\noindent{\emph{Lower bound for $F'(i\mu_k)$.}}
From the definition \eqref{eqn: Fsdef}, \eqref{eqn: spacing} and \eqref{eqn: bdbelow} we have, for $\epsilon>0$ arbitrary,
\begin{align}
|F'(i\mu_k)|&=\frac{|\langle\mathbf{1},v_0(M)\rangle|^2}{\mu_k^2(M)}+\sum_{\substack{j=-(N-1)/2\\ j\neq 0}}^{(N-1)/2}\frac{|\langle \mathbf{1},v_j(M)\rangle|^2}{(\mu_k(M)-\lambda_j(M))^2} \nonumber \\
&\ge \frac{|\langle \mathbf{1},v_k(M)\rangle|^2}{(\mu_k-\lambda_k)^2} \nonumber \\
&\ge \frac{(N^{-\epsilon/4})^2}{(\lambda_k(M)-\lambda_{k+1}(M))^2} \nonumber \\
&\ge N^{1-\epsilon}, \label{eqn: derivlowerbd}
\end{align}
for $N$ large enough, with probability $1-o(1)$.

\noindent{\emph{Estimate for $|F(s)|$.}}
Our application of Cauchy's formula below will require an estimate for $|F(s)|$ for $s$ close to $i\mu_k$. By definition,
\[\mu_k(M)\in (\lambda_k(M),\lambda_{k+1}(M)),\]
if $j \neq k, k+1$, and so we have
\begin{equation}\label{eqn: zerolocation}
|\mu_k(M) -\lambda_j(M)|\ge \min\{|\lambda_j-\lambda_k|, |\lambda_j-\lambda_{k+1}|\} \gtrsim N^{-1/2}, 
\end{equation} 
whenever \eqref{eqn: spacing} holds. Letting
\[z=s-i\mu_k,\]
it follows by \eqref{eqn: zerolocation} that for each $\epsilon'>0$ and $|z|=\lambda= N^{-1/2-c}$ with $c>\epsilon'$, 
\[|s-i\lambda_j(M)|\ge N^{-1/2-\epsilon'}-\lambda,\]
for all $j$, with probability $1-o(1)$.

Let 
\[U_0= \{j: N^{1/2}|\lambda_j-\mu_k|\le N^{-\epsilon'/2}\},\]
and for $l=1,\ldots, \lfloor (1-\epsilon/2) \log_2 N\rfloor$,
\[U_l= \{j:  2^{l-1}N^{\epsilon'/2}\le N^{1/2}|\lambda_j-\mu_k|\le 2^lN^{\epsilon'/2}\}.\]
Writing,
\begin{equation} \label{eqn: absolute}
|F(s)|\le \sum_{j\in U_0}\frac{|\langle \mathbf{1},v_j(M)\rangle|^2}{|\lambda_j-s|}+\sum_{l\ge 1}\sum_{j\in U_l}\frac{|\langle \mathbf{1},v_j(M)\rangle|^2}{|\lambda_j-s|},
\end{equation}
the first term in \eqref{eqn: absolute} is bounded by
\[\#U_0 \cdot (N^{-1/2-\epsilon'}-\lambda)^{-1}\max_j |\langle \mathbf{1},v_j(M)\rangle|^2 \le N^{1/2+10\epsilon'}\#U_0,\]
for any $\epsilon'>0$, with probability $1-o(1)$.
By the local semicircle law  \cite[Theorem 2.2]{sc2} with high probability,
\[\#U_0 \lesssim 1.\]
Thus for $|z|\le \lambda$, 
\[\sum_{j\in U_0}\frac{|\langle \mathbf{1},v_j(M)\rangle|^2}{|\lambda_j-s|}\lesssim N^{1/2}.\]
Similarly, for each $l$, we have
\begin{align*}
\sum_{j\in U_l}\frac{|\langle \mathbf{1},v_j(M)\rangle|^2}{|\lambda_j-s|}&\le C2^{-l}N^{1/2+\epsilon/2} \cdot \# U_l \cdot \max_j |\langle \mathbf{1},v_j(M)\rangle|^2\\
& \lesssim N^{1/2}.
\end{align*}
Summing over $l$ \eqref{eqn: absolute} introduces only a logarithmic factor.
In summary, we have, for each $\epsilon, c>0$
\begin{equation}\label{eqn: Fsummary}
\max\{|F(s)|: |s-i\mu_k|\le N^{-1/2-c} \}\le N^{1/2+\epsilon},
\end{equation}
 with probability $1-o(1)$.

\noindent{\emph{Approximation by a linear function.}} By Cauchy's formula about $\mu_k$, we have
\begin{align}
\label{eqn: Texp}
F(s) &= F'(i\mu_k)(s-i\mu_k)+\frac{(s-i\mu_k)^2}{2\pi i}\int_{|\zeta-i\mu_k|=N^{-1/2-c}}\frac{F(\zeta)}{(\zeta-i\mu_k)^2(\zeta-s)}\,\mathrm{d}\zeta\\
&=  F'(i\mu_k)(s-i\mu_k) + R. \nonumber
\end{align}
Here $c>0$ remains to be chosen.

The remainder term in \eqref{eqn: Texp} is bounded as follows:
\[|R| \le N^{1/2+c} \cdot \max_{|\zeta-i\mu_k|=N^{-1/2-c}}|F(\zeta)| \cdot \frac{|s-i\mu_k|^2}{N^{-1/2-c}-|s-i\mu_k|}.\]

Fix $\epsilon>0$. The events \eqref{eqn: Fsummary} and \eqref{eqn: derivlowerbd} hold simultaneously with probability $1-o(1)$.
Letting $|s-i\mu_k|=\lambda=N^{-1/2-\delta}$ with $c<\delta$, we find
\[|R|\le N^{3/2+2c+\epsilon} \cdot \lambda^2 \le 2N^{1/2-2\delta+2c+\epsilon},\]
whenever \eqref{eqn: Fsummary} holds.

The lower bound \eqref{eqn: derivlowerbd} implies that the solution $s_k$ of the linear equation
\[F'(i\mu_k)(s-i\mu_k)=1\]
satisfies
\[|s_k-i\mu_k|\le N^{-1+\epsilon}.\]
Choosing $c<\epsilon/4$ and $\delta=1/2-\epsilon$, we have the estimate
\[|R| \le N^{-1/2+\epsilon/2}\]
with probability $1-o(1)$, for the remainder in \eqref{eqn: Texp}.

By Rouch\'e's theorem, the functions $F'(i\mu_k)(s-i\mu_k)-1$ and $F(s)-1$ have the same number of zeros inside the circle $|s-i\mu_k|=N^{-1+\epsilon}$ for $N$ sufficiently large. This holds with probability $1-o(1)$ for an index $k$ in $[\alpha(N-1)/2,(1-\alpha)(N-1)/2]$, $\alpha<1$.
\end{proof}
\section{Anti-concentration of eigenvectors}
\label{bysection}
In this section we prove Proposition \ref{prop: lowerbd}. For this purpose we will need a version of a result of Bourgade and Yau \cite[Theorem 1.2]{bourgadeyau}
adapted to anti-symmetric matrices:
\begin{lemma}[Bourgade-Yau for anti-symmetric ensemble]\label{eqn: quelemma}
Let $M_N$ be a sequence of anti-symmetric Hermitian matrices of the form \eqref{eq: Mdef}, and let 
\[\lambda_j(M_N),\quad  j=1,\ldots, (N-1)/2,\] 
denote the $(N-1)/2$ largest eigenvalues, in increasing order. Denote by $v_j(M_N)$ an eigenvector associated to the $j$th positive eigenvalue of $M_N$. Then, for any $j\in [1,\frac{N-1}{2}]$
\begin{equation}
\sqrt{2N}|\langle \mathbf{u}, v_j(M_N) \rangle|\rightarrow |X+iY|,\label{eqn: byclt}
\end{equation}
\emph{uniformly in} $j\in [1,\frac{N-1}{2}]$, where $X$ and $Y$ denote independent standard Gaussian random variables.
\end{lemma}

Assuming Lemma \ref{eqn: quelemma}, the derivation of Proposition \ref{prop: lowerbd} is elementary:
\begin{proof}[Proof of Proposition \ref{prop: lowerbd}]
By \eqref{eqn: byclt}, we have
\begin{equation}\label{eqn: normalconv}
\mathbf{P}(\sqrt{2 N}|\langle \mathbf{u}, v_j \rangle| \ge \epsilon) \rightarrow 2(1-\Phi(\epsilon))
\end{equation}
for $\epsilon>0$. Here we have denoted
\[ \Phi(t) =\frac{1}{\sqrt{2\pi}}\int_{-\infty}^t e^{-\frac{x^2}{2}}\,\mathrm{d}x.\]
For $\eta>0$, and any sequence $(a_N)_{N\ge 0}$ decreasing to $0$ and $N$ sufficiently large, we obtain
\begin{align*}
&\mathbf{P}(\sqrt{N}|\langle \mathbf{u}, v_j \rangle| \ge a_N)\\
\ge &\mathbf{P}(\sqrt{N}|\langle \mathbf{u}, v_j \rangle| \ge \epsilon)\\
\ge &1-\eta/2, 
\end{align*}
where the $\epsilon$ in \eqref{eqn: normalconv} is chosen so as to make $\Phi(\epsilon)<1/2+\eta/4$. 
\end{proof}

The methods of Bourgade and Yau \cite{bourgadeyau} treat general Hermitian or symmetric matrices, but they require a non-degeneracy condition on the variance of the real and imaginary parts. This does not apply in our case. Indeed, as was already seen in Section \ref{sec: antigaussian}, the eigenvector structure for the anti-symmetric Gaussian model differs from that of the GUE. However, once the eigenvector moment flow corresponding to the Hermitian or symmetric Dyson Brownian motion is replaced by the anti-symmetric flow \eqref{eqn: ev0DBM}, \eqref{eqn: evDBM}, the proof of Lemma \ref{eqn: quelemma} very closely follows that of \cite[Theorem 1.3]{bourgadeyau}. The argument merely involves different computations. We present the relevant computations, as well as the necessary background from \cite{bourgadeyau}, in Appendix \ref{sec: by}.

\section{Proof of Lemma \ref{lma: separation}}\label{sec: sep-lma}

\begin{proof} At the zero $i\mu_k$, the equation $F(i\mu_k)=0$ gives
\begin{equation}\label{eqn: distance}
\frac{|\langle \mathbf{1}, v_k(M)\rangle|^2}{\mu_k-\lambda_k(M)}+\frac{|\langle \mathbf{1}, v_{k+1}(M)\rangle|^2}{\mu_k-\lambda_{k+1}(M)}=-\sum_{j\neq k, k+1}\frac{|\langle \mathbf{1},v_j(M)\rangle|^2}{\mu_k-\lambda_j(M)}.
\end{equation}
To use \eqref{eqn: distance} to estimate the distances $|\mu_k-\lambda_k|$ and $|\mu_k-\lambda_{k+1}|$, we need both an upper bound for the quantities $|\langle \mathbf{1},v_i(M)|$ on the right, and lower bounds for $|\langle \mathbf{1},v_k(M)\rangle|$, $|\langle \mathbf{1}, v_{k+1}(M)\rangle|$ on the left side. These are provided by the estimates in Section \ref{sec: innerprod}. 

\noindent \emph{Upper bound for the for the RHS of \eqref{eqn: distance}.} Returning to \eqref{eqn: distance}, we can estimate the right side by
\begin{equation}
\label{eqn: estimate}
\sum_{j: |k- j|\ge N^{\epsilon/2}} \frac{|\langle \mathbf{1},v_j(M)\rangle|^2}{|\mu_k-\lambda_j(M)|}+ \sum_{j: |k - j|\le N^{\epsilon/2}} \frac{|\langle \mathbf{1},v_j(M)\rangle|^2}{|\mu_k-\lambda_j(M)|},
\end{equation}
where in both sums the indices are also restricted to $j\neq k, k+1$.

Combining \eqref{eqn: zerolocation} and
 \eqref{eqn: isotropic}, we find that the second term in \eqref{eqn: estimate} is bounded by
\begin{equation}
(N^{\epsilon/2}+1)\frac{N^\epsilon}{\min\{\mu_k(M)-\lambda_k(M),\lambda_{k+1}(M)-\mu_k(M)\}} \le C(\epsilon)N^{1/2+\epsilon},
\end{equation}
for $\epsilon>0$ arbitrary, with probability $1-o(1)$.
By a dyadic decomposition and the local semicircle law, we obtain the same bound for the first sum in \eqref{eqn: estimate}. We have obtained the estimate
\begin{equation}\label{eqn: Fbound}
\sum_{j: |k- j|\ge N^{\epsilon/2}} \frac{|\langle \mathbf{1},v_j(M)\rangle|^2}{|\mu_k-\lambda_j(M)|}+ \sum_{j: |k - j|\le N^{\epsilon/2}} \frac{|\langle \mathbf{1},v_j(M)\rangle|^2}{|\mu_k-\lambda_j(M)|}\lesssim N^{1/2}.
\end{equation}

\noindent{\emph{Estimate for $|\mu_k-\lambda_k|$ and $|\mu_k-\lambda_{k+1}|$.}} To finish the proof, we will use Lemma \ref{spacingprop}. Returning to \eqref{eqn: distance}, the upper bound \eqref{eqn: Fbound} now implies
\begin{equation}\label{eqn: advanced}
-CN^{1/2+\epsilon} \le \frac{|\langle \mathbf{1}, v_k(M)\rangle|^2}{\mu_k(M)-\lambda_k(M)}+\frac{|\langle \mathbf{1}, v_{k+1}(M)\rangle|^2}{\mu_k-\lambda_{k+1}(M)} \le CN^{1/2+\epsilon}.
\end{equation}
Rearranging this inequality (noting that $\mu_k(M)-\lambda_{k+1}(M) <0$), we obtain:
\begin{equation} \label{eqn: boundtheratio}
-\frac{\mu_k(M)-\lambda_{k+1}(M)}{\mu_k(M)-\lambda_k(M)} \le -CN^{1/2+\epsilon}(\mu_k(M)-\lambda_{k+1}(M))+\frac{|\langle \mathbf{1}, v_{k+1}(M)\rangle|^2}{|\langle \mathbf{1}, v_{k}(M)\rangle|^2}.
\end{equation}
By Proposition \ref{spacingprop}, the first term on the right of \eqref{eqn: boundtheratio} is bounded by
\begin{align*}
-CN^{1/2+2\epsilon}(\mu_k(M)-\lambda_{k+1}(M))&\le -CN^{1/2+2\epsilon}(\lambda_k(M)-\lambda_{k+1}(M))\\
&\le -CN^{3\epsilon},
\end{align*}
with probability $1-o(1)$. The second term in \eqref{eqn: boundtheratio} is controlled using both \eqref{eqn: isotropic} and \eqref{eqn: bdbelow}:
\[\frac{|\langle \mathbf{1}, v_{k+1}(M)\rangle|^2}{|\langle \mathbf{1}, v_{k}(M)\rangle|^2}\le N^\epsilon\]
with probability $1-o(1)$. Thus, we have
\[\left|\frac{\mu_k(M)-\lambda_{k+1}(M)}{\mu_k(M)-\lambda_k(M)}\right|\le N^{3\epsilon}.\]

Repeating the argument in the last paragraph starting from \eqref{eqn: advanced}, but multiplying by $\mu_k(M)-\lambda_k(M)$ instead, we obtain:
\begin{equation}\label{eqn: stableratio}
N^{-3\epsilon}\le \left|\frac{\mu_k(M)-\lambda_{k+1}(M)}{\mu_k(M)-\lambda_k(M)}\right|\le N^{3\epsilon}
\end{equation}
with probability $1-o(1)$. By \eqref{eqn: spacing}, we have
\[(\lambda_{k+1}(M)-\mu_k(M))+(\mu_k(M)-\lambda_k(M))= \lambda_{k+1}(M)-\lambda_k(M) \ge N^{-1/2+\epsilon}\]
with probability $1-o(1)$, so
\begin{align*}
 &\min\{\mu_k(M)-\lambda_k(M),\lambda_{k+1}(M)-\mu_k(M)\}\\
\ge &\ N^{-3\epsilon}\max\{\mu_k(M)-\lambda_k(M),\lambda_{k+1}(M)-\mu_k(M)\}\\
\ge &\ N^{-3\epsilon}N^{-1/2-\epsilon}/2.
\end{align*}
That is, \eqref{eqn: zerolocation} holds also for $j=k$ and $j=k+1$ with probability $1-o(1)$ as $N\rightarrow \infty$.
\end{proof}

\section{Proof of Proposition \ref{spacingprop}}\label{sec: repulse}
Our proof uses results from \cite{ey}. (See Section \ref{sec: gapsec}.)
These rely on previous ``rigidity'' theorems  for $\beta$-ensemble measures; see \cite{betauniversality,edgeuniversality}. Attempting to reproduce these arguments for the ensemble \eqref{eq: pdf} leads to technical complications in case $N$ is odd, due to the appearance of a logarithmic term in the potential because of the eigenvalue at 0.
To circumvent this difficulty, we compare a minor of $M$ to an $N-1$-dimensional Gaussian matrix. When $N$ is even, there is no logarithmic term in the potential of the anti-symmetric Gaussian ensemble \eqref{eq: pdf}.

\begin{proof}[Proof of Proposition \ref{spacingprop}]Let $M^{(1)}$ be the $(1,1)$ minor of the matrix $M$, that is, the $(N-1)$-dimensional square matrix obtained by removing the first row and column of $M$.

Concerning the eigenvalues of the minor $M^{(1)}$, the spacing information we need is:
\begin{lemma}\label{minorspacinglemma}
For each $j$, we have
\begin{equation}\label{eqn: minorspacing}
|\lambda_{i+j}(M^{(1)})-\lambda_{i+j+1}(M^{(1)})|\asymp N^{-1/2},
\end{equation}
for $j=0,\ldots, n-1$.
\end{lemma}
Lemma \ref{minorspacinglemma} is proved proved in Section \ref{sec: gapsec}.

Assume that the first entry of the $(i+j)$th eigenvectors is non-zero
\begin{equation}\label{eqn: nonzerooverlap}
\langle v_{i+j}(M), \mathbf{e}_1\rangle \neq 0
\end{equation}
for $j=0,\ldots n-1$. Then the $(1,1)$ entry of the resolvent of $M$ is
\[R_{11}(s)=\sum_{j=1}^N\frac{|\langle v_j(M),\mathbf{e}_1\rangle|^2}{\lambda_j(M)-s}.\]
By Schur's complement, we can express $R_{11}$ as
\begin{equation}\label{eqn: Rinv}
R_{11}(s)^{-1} = M_{11}-s-\sum_{j=1}^{N-1}\frac{|\langle v_j(M^{(1)}), H \rangle|^2}{\lambda_j(M^{(1)})-s},
\end{equation}
with
\[H=(M_{12},\dots,M_{1,N-1}).\]
The Cauchy interlacing theorem implies that
\[\lambda_i(M^{(1)})\le \lambda_i(M)\le \lambda_{i+1}(M^{(1)}).\]
Consider the function
\[s\mapsto F^{(1)}(s)=s-\sum_{j=1}^{N-1}\frac{|\langle v_j(M^{(1)}), H \rangle|^2}{s-\lambda_j(M^{(1)})}.\]
By equation \eqref{eqn: Rinv}, the eigenvalue $\lambda_{i+j}(M)$ is located at the solution of the equation
\begin{equation}\label{eqn: M11}
F^{(1)}(s)=M_{11}
\end{equation}
along the real axis lying between the poles at $\lambda_{i+j}(M^{(1)})$ and $\lambda_{i+j+1}(M^{(1)})$. Conditioning on the vector $H$, we can again use the method of proof of Proposition \ref{prop: lowerbd} to show
\begin{equation}\label{eqn: minorlwrbd}
|\langle v_k(M^{(1)}), H \rangle|\gtrsim 1, \text{ for } k=i,\ldots,i+n.
\end{equation}
The equations for the eigenvectors in the even-dimensional case differ from \eqref{eqn: ev0DBM}, \eqref{eqn: evDBM}. There is no term in \eqref{eqn: evDBM} corresponding to the real eigenvector associated with the eigenvalue 0. Correspondingly, the equation \eqref{eqn: ev0DBM} is absent from the system. However, the implementation of the method is identical.

Given \eqref{eqn: minorlwrbd}, the function $F^{(1)}(s)$ is monotone on each interval between its poles at $\lambda_{i+j}(M^{(1)})$, $j=1,\ldots, n-1$, and so \eqref{eqn: M11} has a unique solution in each of these intervals: this solution is an eigenvalue $\lambda_{k}(M)$. The equation $F^{(1)}(s)=0$ similarly has a unique solution $\mu_{k}(M^{(1)})$ in this interval. Noting that
\[|M_{11}|=1,\]
we repeat the argument in the proof of Proposition \ref{prop: exist} to show
\begin{equation}
\label{closetozero}
|\mu_{k}(M^{(1)})-\lambda_k(M^{(1)})|\lesssim N^{-1}.
\end{equation}
Combining \eqref{closetozero} and \eqref{eqn: minorspacing}, we obtain Proposition \ref{spacingprop}.
\end{proof}

\subsection{Gap universality for anti-symmetric matrices}\label{sec: gapsec}
It remains to prove Lemma \ref{minorspacinglemma}. We will show that the main result in \cite{ey}, Theorem 2.2, proved for Wigner matrices with non-degenerate real and imaginary parts, can be extended to anti-symmetric matrices of even size such as $M^{(1)}$.

Consider the even dimensional Gaussian model, given by the measure
\begin{equation}\label{eqn: antiGUEeven}
\mu_{N-1}(\mathrm{d}x)=\frac{1}{Z'_{N-1}}\prod_{1\le j < k\le (N-1)/2}|\lambda_i^2-\lambda_j^2|^2\prod_{i=1}^{(N-1)/2}e^{-(N-1)\frac{\lambda_i^2}{2}}
\end{equation}
on $\Sigma_N^+=\{0<\lambda_1<\ldots <\lambda_{\frac{N-1}{2}}\}$. This is the joint eigenvalue distribution of $iK'/\sqrt{N-1}$, where $K'$ is an $(N-1)\times(N-1)$ anti-symmetric matrix from the Gaussian model \eqref{eqn: antiguemodel}, with $N$ odd. Compare this with \eqref{eq: pdf}.

Let $N$ be odd, and $M'$ be an $(N-1)\times (N-1)$ anti-symmetric Hermitian matrix, that is $(M')_{ij} = -(M')_{ji}$, and $\Re M_{ij} =0$ for $1\le i,j\le N$. We assume the entries are independent, identically distributed, with variance
$1$. As previously, we label the $N-1$ eigenvalues increasing order as
\[\lambda_{-(N-1)/2}(M')\le \lambda_{-(N-1)/2+1}(M') \le \ldots \le \lambda_{(N-1)/2-1}(M')\le \lambda_{(N-1)/2}(M').\]
Recall that $\lambda_{-j}(M)=-\lambda_{j}(M')$. From now on, we restrict our attention to the $(N-1)/2$ leading eigenvalues. The main result of this section is the following adaptation of \cite[Theorem 2.2]{ey}, proved in Section \ref{sec: gapsec}:
\begin{theorem} \label{ourgapthm}
Fix positive numbers $\alpha$, $O_\infty$ and $O_{supp}$, and $n\in \mathbf{N}$. Consider a sequence $O_N$ of $C^\infty(\mathbf{R}^n)$ $n$-variables observables, satisfying
\[ \sup_n \|O_N\|_\infty \le O_\infty, \quad
\mathrm{supp}\, O_N \subset [-O_{supp},O_{supp}]^n.\] 

Let $\lambda_j := \frac{1}{\sqrt{N-1}}\lambda_j(M')$. Then, for $O=O_N$ and any 
\[j\in [\alpha (N-1)/2,(1-\alpha)(N-1)/2],\] 
we have
\begin{equation}
\big| (\mathbf{E}-\mathbf{E}^{\mu_{N-1}})O(N(\lambda_j-\lambda_{j+1}),\ldots,N(\lambda_j-\lambda_{j+n}))\big| \le CN^{-\epsilon}\|O'\|_\infty,
\end{equation}
for $N\ge N_0(n,\alpha,O_\infty,O_{supp})$.
\end{theorem}

Assuming Theorem \ref{ourgapthm}, we can give the
\begin{proof}[Proof of Lemma \ref{minorspacinglemma}]
Choosing $0\le O(x) \le 1$ to be equal to $1$ for $N^{-1/2-\epsilon/2} \le x \le N^{-1/2+\epsilon/2}$ and zero for $x\notin [(1/2)N^{1/2-\epsilon/2},(3/2)N^{1/2+\epsilon/2}]$, we find
\begin{align*}
&\mathbf{P}(N^{-1/2-\epsilon}\le |\lambda_{i+j}-\lambda_{i+j+1}| \le 2N^{-1/2+\epsilon})\\
\ge&\mathbf{E}^{\mu_{N-1}}Q(N^{-1/2-\epsilon}\le |\lambda_{i+j}-\lambda_{i+j+1}| \le 2N^{-1/2+\epsilon})-o(1) \\
= &1-o(1),
\end{align*}
yields \eqref{eqn: minorspacing} when applied to finitely many indices $j$.
\end{proof}

\subsubsection{Gap universality: Proof of Theorem \ref{ourgapthm}}
If $M'$ were a Wigner matrix in the GUE class, that is, if the real part were non-degenerate, Theorem \ref{ourgapthm} would be identical to (in fact, somewhat weaker) \cite[Theorem 2.2]{ey}. The proof of Theorem \ref{ourgapthm} is almost the same, but since the setting is slightly different, we will explain the necessary modifications in some detail.

The proof of \cite[Theorem 2.2]{ey} combines a Green's function comparison argument with an analysis of the dynamics of Dyson Brownian motion in terms of local measures (See \cite[Section 4.1]{ey}). These are conditioned versions of \eqref{eqn: antiGUEeven}: split the variables $\lambda_1,\ldots, \lambda_{(N-1)/2}$ as
\[(\lambda_1,\ldots , \lambda_{(N-1)/2})= (y_1,\ldots, y_{L-K-1},x_{L-K},\ldots, x_{L+K}, y_{L+K+1},\ldots, y_{(N-1)/2}).\]
Let 
\[\mathbf{x}=(x_{L-K}, \ldots, x_{L+K}) \text{ and } \mathbf{y}=(y_1,\ldots,y_{L-K-1},y_{L+K+1},\ldots,y_N).\]
$L$ and $K$ are parameters defined by
\begin{equation*}
L \in [\alpha (N-1)/2, (1-\alpha)(N-1)/2]\cap \mathbf{Z}_+, \quad N^\delta \le K\le N^{1/4}.
\end{equation*}
The local measure  $\mu_{\mathbf{y}}$ with \emph{boundary condition} $\mathbf{y}$ is defined by
\begin{equation}\label{eqn: localdef}
\mu_{\mathbf{y}}(\mathrm{d}\mathbf{x})=\mu_{\mathbf{y}}(\mathbf{x})\,\mathrm{d}\mathbf{x} \quad \text{ and } \quad \mu_{\mathbf{y}}(\mathbf{x})=\mu(\mathbf{x},\mathbf{y})\left[\int \mu(\mathbf{y},\mathbf{x})\,\mathrm{d}\mathbf{x}\right]^{-1}.
\end{equation}
The measure $\mu_{\mathbf{y}}$ also has the representation \cite[Eqs. (4.5)-(4.6)]{ey}
\begin{align}
\mu_{\mathbf{y}}&=Z^{-1}_{\mathbf{y}}e^{-N\mathcal{H}_{\mathbf{y}}}, \label{eqn: localmeasureform}\\
\mathcal{H}_{\mathbf{y}}(\mathbf{x})&=\sum_{i\in I} \frac{1}{2}V_{\mathbf{y}}(x_i)-\frac{2}{N-1}\sum_{i,j\in I, i<j}\left(\log|x_i-x_j|+\log|x_i+x_j|\right), \label{eqn: hamiltonian}\\
V_{\mathbf{y}}(x)&=\frac{x^2}{2}-\frac{4}{N-1}\sum_{j\notin I}(\log|x-y_j|+\log|x+y_j|).
\end{align}

The arguments in \cite[Section 6]{ey} leading to the proof of \cite[Theorem 2.2]{ey} can be reproduced word for word to prove Theorem \ref{ourgapthm}, provided the dynamics \cite[Equations (6.6), (6.7)]{ey} is replaced by the (odd) anti-symmetric Dyson Brownian motion:
\begin{align}
\partial_t f_t &= \mathcal{L}f_t\\
\mathcal{L}&:= \sum_{i=1}^N \frac{1}{N-1}\partial_i^2+\sum_{i=1}^{(N-1)/2} \left(-\frac{1}{2}\lambda_i+\frac{2}{N-1}\sum_{j\neq i}\frac{1}{\lambda_i-\lambda_j}+\frac{2}{N-1}\sum_{j\neq i}\frac{1}{\lambda_i-\lambda_j}\right)\partial_i.
\end{align}
on $\Sigma_{N-1}^+$. The proof requires three inputs:
\begin{enumerate}
\item Rigidity for the global measure $\mu_{N-1}$.
\item Rigidity and level repulsion for the local measures $\mu_{\textbf{y}}$ \eqref{eqn: localdef}.
\item Gap universality, in the sense of \cite[Theorem 4.1]{ey}, for local measures associated to the Gaussian ensemble \eqref{eqn: antiGUEeven}.
\end{enumerate}

\paragraph{Rigidity for the global measure.} Let us first address item (1). Rigidity refers to the fact that the individual eigenvalue locations under \eqref{eqn: antiGUEeven} are very close to those predicted by the limiting spectral distribution. For \eqref{eqn: antiGUEeven}, the limiting distribution for the positive eigenvalues is the quarter-circle distribution:
\[\mu_{qc}(\mathrm{d}x) = 2\rho_{sc}(\mathrm{d}x)\mathbf{1}_{[0,2]}(x)\,\mathrm{d}{x}.\]
The classical locations are determined by the relation
\[\int_0^{\gamma_j}\rho_{sc}(\mathrm{d}x)= \frac{j}{(N-1)/2}.\]
The global rigidity statement required in \cite{ey} (See Eqn. 4.7, and Section 5.1) is that, for any fixed $\alpha>0$, $\nu>0$, there exist constants $C_0,c_1,c_2>0$, such that for any $N$ large enough:
\begin{equation}
\mathbf{P}^{\mu_{N-1}}(|\gamma_k-\lambda_k|>N^{-1+\nu})\le C_0\exp(-c_1N^{c_2}), \label{eqn1}
\end{equation}
for $k\in [\alpha(N-1)/2,(1-\alpha)(N-1)/2]$,
\begin{equation}\
\mathbf{P}^{\mu_{N-1}}(|\gamma_k-\lambda_k|>N^{-4/15+\nu})\le C_0\exp(-c_1N^{c_2}), \label{eqn2}
\end{equation}
for $N^{3/5}\le k \le (N-1)/2-N^{3/5+\nu}$, and
\begin{equation}
\mathbf{P}^{\mu_{N-1}}(|\gamma_k-\lambda_k|>C)\le C_0\exp(-c_1N^{c_2}) \label{eqn3}
\end{equation}
for all $0\le k\le (N-1)/2$. These three estimates (indeed, a much stronger result) follows from the local semi-circle law applied to the Gaussian matrix model for \eqref{eqn: antiGUEeven}, see \cite[Section 7, Theorem 7.6]{sc1}.

\paragraph{Rigidity for the local measures $\mu_{\mathbf{y}}$.}
This is the statement that, for each $k\in I_{L,K}=[L-K,L+K]\cap \mathbf{Z}_+$ and any fixed $\xi>0$:
\begin{equation}
\label{eqn: subg}
\mathbf{P}^{\mu_{\mathbf{y}}}(|x_k-\alpha_k|\ge uK^\xi N^{-1})\le Ce^{-cu^2}.
\end{equation}
(See \cite[Theorem 4.2]{ey}.) 
Here, 
\[\alpha_k = \frac{1}{2}(y_{L-K-1}+y_{L+K+1})+\frac{j-L}{2K+2}, \quad j \in I_{L,K}.\]
The proof of \cite[Theorem 4.2]{ey} in \cite[Section 7.1]{ey} depends only on the convexity of the Hamiltonian \eqref{eqn: hamiltonian} for boundary conditions $\mathbf{y}$ in a set of overwhelming probability. This convexity is not affected by the additional $\log|x_i+x_j|$ term. Level repulsion, the second part of item (2) above, is the statement that, for boundary conditions $\mathbf{y}$ satisfying the rigidity conditions \eqref{eqn1}, \eqref{eqn2}, \eqref{eqn3}, we have:
\[\mathbf{P}^{\mu_{\mathbf{y}}}(x_{i+1}-x_i \le s/N) \le C(Ns)^{\beta+1}, \quad  i\in [L-K-1,L+K]\cap \mathbf{Z}_+,\]
and moreover
\[\mathbf{P}^{\mu_{\mathbf{y}}}(x_{i+1}-x_i\le s/N)\le C(K^\xi s\log N)^{\beta+1}, \quad i\in [L-K-1,L+K]\cap \mathbf{Z}_+,\]
if 
\[\mathbf{P}^{\mu_{\mathbf{y}}}(|x_k-\alpha_k|\ge CK^\xi N^{-1})\le C\exp(-K^c)\]
for some constants $C$, $c$ and any $k\in I$. The proof for the anti-symmetric ensemble is identical to the one given in \cite[Section 7.1]{ey}.

\paragraph{Gap universality for local measures via H\"older regularity.} The proof of item (3) relies on items (1) and (2), and is the main result in \cite{ey}. It follows directly from a comparison result between two local measures of the form \eqref{eqn: localmeasureform}, on the interval $J$ and the same subset of indices $I$, with different potentials $V_{\mathbf{y}}$ and $V_{\mathbf{y}'}$. In the case that concerns us, the two potentials are Gaussian with $\beta=2$, and differ only through their boundary conditions. The starting point is an interpolation between observables with respect to $\mu_{\mathbf{y}}$ and with respect to $\mu_{\tilde{\mathbf{y}}}$:
\begin{align}
\mathbf{E}^{\mu_{\tilde{\mathbf{y}}}}F(\mathbf{x})-\mathbf{E}^{\mu_{\mathbf{y}}}F(\mathbf{x}) &= \int_0^1 \beta \langle h_0(\mathbf{x}); Q(\mathbf{x})\rangle_{\omega^r_{\mathbf{y},\tilde{\mathbf{y}}}}\,\mathrm{d}r, \label{eqn: interpolation-integral}\\
\omega^r_{\mathbf{y},\tilde{\mathbf{y}}} &= Z_r^{-1} e^{-\beta r(V_{\tilde{\mathbf{y}}}-V_{\mathbf{y}})(\mathbf{x})}\mu_{\mathbf{y}},\\
h_0(\mathbf{x})&=\sum_{i\in I} (V_{\mathbf{y}}(x_i)-V_{\tilde{\mathbf{y}}}(x_i)).
\end{align}
The idea in \cite{ey} is to use a representation for the integrand in \eqref{eqn: interpolation-integral} as a random walk in a random environment to show that it is small, uniformly in $0\le r \le 1$ for observables $F=O$, where 
\[O(x_p-x_{p+1},\ldots, x_p-x_{p+n})\]
is a function of particle gaps. Indeed, by \cite[Theorem 8.1]{ey}, there are constants $\epsilon>0$ and $C>0$ such that
\[|\langle h_0; O(x_p-x_{p+1},\ldots,x_p-x_{p+n})\rangle_{\omega^r_{\mathbf{y},\tilde{\mathbf{y}}}}|\le K^{C\xi}K^{-\epsilon}\|O'\|_{\infty},\]
for fixed $\xi^*>0$, $\xi$ sufficiently small and $K$ large enough.

 In the case where $O$ depends on a single gap, the random walk representation is
\[
\langle h_0; O(x_p-x_{p+1})\rangle_\omega = \frac{1}{2}\int_0^\infty \int \sum_{b\in I} \partial_b h_o(x) \mathbf{E}_x\left[O'(x_p-x_{p+1})\left(v_p^b(\sigma)-v_{p+1}^b(\sigma)\right)\right]\omega(\mathrm{d}\mathbf{x}),
\]
where we have written $\omega=\omega_{\mathbf{y},\tilde{\mathbf{y}}}$. To define $\mathbf{v}(\sigma)$, we first introduce $\mathbf{x}(s)$, $s\ge 0$, the solution of the stochastic dynamics determined by
\[\mathrm{d}x_i = \mathrm{d}B_i + 2\partial_i \mathcal{H}_{\mathbf{y},\tilde{\mathbf{y}}}\mathrm{d}t\]
with initial condition $\mathbf{x}(0)=\mathbf{x}$. $\mathcal{H}_{\mathbf{y},\tilde{\mathbf{y}}}$ is the Hamiltonian associated with the interpolation measure:
\[\mathrm{d}\omega^r_{\mathbf{y},\tilde{\mathbf{y}}}= Z_\omega^{-1} e^{-2\mathcal{H}_{\mathbf{y},\tilde{\mathbf{y}}}(\mathbf{x})}\,\mathrm{d}\mathbf{x}. \]
For a fixed path $\{\mathbf{x}(s): s\ge 0\}$, define the matrix
\begin{align*}
\mathcal{A}(s)&=\tilde{\mathcal{A}}(\mathbf{x}(s))\\
\tilde{\mathcal{A}}&=2\mathrm{Hess}\mathcal{H}_{ \mathbf{y},\tilde{\mathbf{y}} }(\mathbf{x}).
\end{align*}
$\mathbf{v}^b(t,\mathbf{x}(t))$ is defined as the solution of the equation 
\begin{equation}
\label{eqn: degiorgiequation}
\partial_t \mathbf{v}^b(t) = -\mathcal{A}(t)\mathbf{v}^b(t), \quad v^b_a(0)=\delta_{ba}(0).
\end{equation}
The core of \cite{ey} is an analysis of the parabolic evolution with time-dependent random coefficients \eqref{eqn: degiorgiequation}. In addition to the probabilistic properties of the paths $\mathbf{x}(s)$ which are derived from the rigidity and level repulsion estimates (item (2)) above (see \cite[Section 9.3]{ey}), the main input is the elliptic structure of the matrix 
\[\tilde{\mathcal{A}}=\tilde{B}+\tilde{W},\]
where 
\[[\tilde{B}(\mathbf{x})\mathbf{v}]_j = -\sum_k \tilde{B}_{jk}(\mathbf{x})(v_k-v_j)\]
is the part of the Hessian of the Hamiltonian coming from the logarithmic interaction:
\begin{equation}\label{eqn: Bbounds}
\tilde{B}_{jk}(\mathbf{x}) =\frac{2}{(x_j-x_k)^2}+\frac{2}{(x_j+x_k)^2}\ge 0.
\end{equation}
The proofs in Sections 8, 9 and 10 of \cite{ey} are unaffected by the additional term $1/(x_j+x_k)^2$ in \eqref{eqn: Bbounds}, since this term is of lower order than $1/(x_j-x_k)^2$:
\[x_j,x_k \in [0,2] \Rightarrow \frac{1}{(x_j+x_k)^2} \le \frac{1}{(x_k-x_j)^2}.\]
Given that the rigidity and level repulsion estimates remain the same whether this term is present in the logarithmic interaction or not, we can follow all the steps in \cite{ey} to establish for the anti-symmetric ensemble \eqref{eqn: antiGUEeven}, and obtain Theorem \ref{ourgapthm}.

\section{The GUE case}\label{sec: gueext}
The method of proof in the Section \ref{sec: inter2} can be applied to perturbations of a GUE matrix. The result is as follows:
\begin{theorem}
Let $G$ be an $N\times N$ GUE matrix, and let $\mathbf{b}_N \in \mathbf{R}^N$ be a unit vector.
Consider the matrix
\[W = G + iN\cdot \mathbf{b}_N\mathbf{b}_N^\intercal.\]
Let $0<\alpha<1$. Then there is $\epsilon>0$ such that, for each $i\in [\alpha N,(1-\alpha)N]$, and each $j=0,\ldots n-1$, $W$ has an eigenvalue (which we label $\lambda_{i+j}(W)$) with
\[\lambda_{i+j}(G) < \Re \lambda_{i+j}(W) < \lambda_{i+j+1}(G),\]
with probability $1-N^{-\epsilon}$.
\begin{proof}
Let $X=(X_1,\ldots,X_N)^\intercal$ be a complex Gaussian vector with mean 0 and covariance $I$. Then $UX$ is Gaussian vector with the same distribution for any unitary transformation $U$. From this, it follows easily that
\[Y=\frac{X}{\|X\|_2}\]
has uniform distribution on the unit sphere in $\mathbf{C}^N$. In particular, if $v$ is any eigenvector of the GUE matrix $G$, the distributions of $\langle v,\mathbf{b}_N \rangle$ and
$\langle Y, \mathbf{b}_N\rangle$ are identical. Thus
\begin{align*}
\mathbf{P}(\sqrt{N}|\langle v, \mathbf{b}_N\rangle|\le \eta)&= \mathbf{P}(\sqrt{N}|\langle Y, \mathbf{b}_N\rangle|\le \eta)\\
&= \mathbf{P}( \sqrt{N}|X_1|\le \eta\|X\|_2)\\
&\le \mathbf{P}(|X_1|\le \eta N^{\epsilon})+\mathbf{P}(Z/\sqrt{N} \ge N^{\epsilon}),
\end{align*}
where $Z$ has the chi distribution with $N$ degrees of freedom. Choosing $\eta \le N^{-2\epsilon}$, we obtain
\begin{equation}
\mathbf{P}(\sqrt{N}|\langle v, \mathbf{b}_N\rangle|\le N^{-\epsilon}) \le CN^{-\epsilon}.\label{eqn: gueeigenvectbound}
\end{equation}
Replacing \eqref{prop: lowerbd} by \eqref{eqn: gueeigenvectbound} in the proof of Theorem 2, (ii) above to obtain \eqref{eqn: bdbelow}, the result follows. Note also that we do not need to appeal to the results in \cite{bloemendaletal} to obtain the bound, since the distribution of the projections of eigenvectors are known.
\end{proof}
\end{theorem}

\appendix
\section{Bourgade and Yau's eigenvector moment flow}
\label{sec: by}
In \cite{bourgadeyau}, the authors study the evolution under the Hermitian analog of \eqref{eqn: evDBM} \cite[Equation (2.3)]{bourgadeyau} of the joint moments of the quantities
\[w_k=\langle \mathbf{q}, u_k \rangle, \quad k=1,\ldots,N,\]
where $u_k$ are the eigenvectors of an $N\times N$ Hermitian random matrix and $\mathbf{q}$ is a unit vector in $\mathbf{R}^N$.

Define
\[{\mathrm{P}^{(h)}}^{j_1,\ldots,j_m}_{i_1,\ldots,i_m}(w_1,\ldots,w_N) = \prod_{l=1}^m w^{j_l}_{i_l}\bar{w}_{i_l}^{j_l}.\]
(The superscript $(h)$ stands for ``Hermitian.'')
Let $H_0$ be a Hermitian Wigner matrix, and $H(t)$, $t\ge 0$, be the solution of the Hermitian Dyson Brownian motion with initial data $H_0$. Suppose $(w_1(t), \ldots, w_N(t))$ are the solutions to the Hermitian Dyson Brownian motion with initial data $\vec{w}(H_0)$, where $\vec{w}(H_0)$ is an orthogonal set of eigenvectors of $H_0$. 
We introduce the normalized moments, conditional on $H_0$ and the eigenvalues $\vec{\lambda}=(\lambda_1(t),\ldots,\lambda_N(t))$ of $H(t)$:
\begin{equation}\label{Q}
{\mathrm{Q}_{ \vec{\lambda} } }^{j_1,\ldots,j_m}_{i_1,\ldots,i_m}(t)=\mathbf{E}^{H_0}({\mathrm{P}^{(h)}}^{j_1,\ldots,j_m}_{i_1,\ldots,i_m}(t) | \ \vec{\lambda})\prod_{l=1}^m (2^{j_l}j_l!)^{-1}.
\end{equation}
The evolution of the quantities \eqref{Q} is given by
\begin{align}\label{eqn: dQdt}
\partial_t {\mathrm{Q}_{ \vec{\lambda} }(t) } &= \sum_{1\le k< l\le N} c_{kl}(t)V_{kl}\overline{V}_{kl}\mathrm{Q}_{ \vec{\lambda} }(t),\\
V_{kl}&=w_k\partial_{w_l}-\overline{w}_l\partial_{\overline{w}_k},\  \overline{V}_{kl} = \overline{w}_k\partial_{\overline{w}_l}-w_l\partial_{w_k} \nonumber \\
c_{kl}(t)&=\frac{1}{N(\lambda_k(t)-\lambda_l(t))^2}.\nonumber
\end{align} 
Here we have omitted the indices on the moment $\mathrm{Q}_{ \vec{\lambda} }(t)$. 

The action of the terms $V_{kl}\overline{V}_{kl}$ on the moments is given by
\begin{align}
V_{kl}\overline{V}_{kl}{\mathrm{Q}_{ \vec{\lambda} }}^{j_1,\ldots,j_m}_{i_1,\ldots,i_m}(t) &= j_1{\mathrm{Q}_{ \vec{\lambda} }}^{1,j_1-1,\ldots,j_m}_{l,i_1,\ldots,i_m}(t)-j_1{Q_{\vec{\lambda}}}^{j_1,\ldots,j_m}_{i_1,\ldots,i_m}(t), \quad l\notin \{i_1,\ldots,i_m\}, \label{eqn: generator1}\\
V_{kl}\overline{V}_{kl}{\mathrm{Q}_{ \vec{\lambda} }}^{j_1,\ldots,j_m}_{i_1,\ldots,i_m}(t) &= j_1(j_2+1){\mathrm{Q}_{ \vec{\lambda} }}^{j_1-1,j_2+1,\ldots,j_m}_{i_1,\ldots,i_m}(t) \label{eqn: generator2}\\
&\quad +j_2(j_1+1){\mathrm{Q}_{\vec{\lambda}}}^{j_1+1,j_2-1,\ldots,j_m}_{i_1,\ldots,i_m}(t) \nonumber  \\
&\quad -(j_1(j_2+1)+j_2(j_1+1)){\mathrm{Q}_{ \vec{\lambda} }}^{j_1,\ldots,j_m}_{i_1,\ldots,i_m}(t) \nonumber
\end{align}
The relations \eqref{eqn: generator1}, \eqref{eqn: generator2} afford an interpretation of  \eqref{eqn: dQdt} as a multiparticle random walk. The indices
\[\{(i_1,j_1),\ldots,(i_m,j_m)\}\]
on the moments \eqref{Q} are taken to represent configurations of particles at locations $1,...N$, $j_k$ being the number of particles at $i_k\in \{1,...,N\}$. Each such configuration corresponds to a function
\[\mathbf{\eta}:\{1,\ldots, N\} \rightarrow \mathbf{Z}_{\ge 0},\]
where
\begin{align*}
\mathbf{\eta}(i_k)&:=\eta_{i_k}=j_k,\\
\mathbf{\eta}(i)&=0,\quad i\notin \{i_1,\ldots, i_m\}.
\end{align*}
Then, \eqref{eqn: dQdt} describes the evolution along the eigenvector flow, conditioned on the eigenvalues, of the function 
\begin{equation}
f^{(h)}_{\vec{\lambda}}(t,\mathbf{\eta})= {\mathrm{Q}_{ \vec{\lambda} }(t)}^{j_1,\ldots,j_m}_{i_1,\ldots,i_m}(t)
\end{equation}
of configurations with fixed particle number $\mathcal{N}(\mathbf{\eta})=\sum_j \eta_j$. Equations \eqref{eqn: generator1}, \eqref{eqn: generator2} imply that the flow preserves particle number. The equations for Bourgade and Yau's \emph{eigenvector moment flow} are
\begin{equation}
\partial_t f^{(h)}_{\vec{\lambda}}(t,\mathbf{\eta}) = \sum_{i\neq j} c_{ij}(t)\eta_i(1+\eta_j)(f(\mathbf{\eta}^{i,j})-f(\mathbf{\eta})),
\end{equation}
with $\mathbf{\eta}^{i,j}$ the configuration obtained from $\mathbf{\eta}$ by moving a single particle from location $i$ to location $j$.

\cite[Theorem 4.3]{bourgadeyau} shows that for each total particle number, the quantity $f_t(\mathbf{\eta})$ converges to $1$ uniformly, at a polynomial rate with $N$, provided $t\ge N^{-\frac{1}{4}+\epsilon}$. This implies \cite[Corollary 4.4]{bourgadeyau} that all moments of projections of the eigenvectors of matrices with a small Gaussian component converge to the corresponding Gaussian moments, uniformly over moments of any fixed order. This is then extended to general Wigner matrices by a comparison argument \cite[Section 5]{bourgadeyau}, to obtain the result:
\begin{theorem}[Bourgade, Yau]
For any sequence of generalized Hermitian Wigner matrices, fixed index set $I$ and unit vector $\mathbf{q}$ in $\mathbf{R}^N$,
\[\sqrt{2 N} |\langle \mathbf{q},u_k\rangle|)_{k\in I}\rightarrow (|N^{(1)}(0,1)+iN^{(2)}(0,1)|)_{j=1}^m,\]
in the sense of moments, uniformly in $\mathbf{q}$ and $I$ such that $|I|\le m$ fixed. $N^{(1)}$ and $N^{(2)}$ are independent $N(0,1)$ random variables.
\end{theorem}

\subsection{Generator for the anti-symmetric moment flow}
We wish to reproduce the argument in \cite{bourgadeyau} for the anti-symmetric Dyson Brownian motion \eqref{eqn: asdbm}, \eqref{eqn: ev0DBM}, \eqref{eqn: evDBM}.
In this case, the generator of the flow on moments \eqref{eqn: dQdt}, is somewhat more complicated. Letting 
\[f\left(u_i(\alpha)_{1\le i,\alpha \le (N-1)/2}, \overline{u}_i(\alpha)_{1\le i,\alpha \le (N-1)/2},u_0(\alpha)_{1\le\alpha\le (N-1)/2}\right)\]
 be a smooth function of the eigenvector entries. Then, if $F=\mathbf{E}(f\mid \vec{\lambda}(t))$, a computation using Ito's formula similar to \cite[Appendix B]{bourgadeyau} shows
\begin{align}
\partial_t F &= \frac{1}{2}\sum_{1\le i<j\le (N-1)/2} c_{ij}(t) (X_{ij}\overline{X_{ij}}+\overline{X_{ij}}X_{ij}) F\\
& +\frac{1}{2}\sum_{1\le i<j\le (N-1)/2}\tilde{c}_{ij}(t) (Y_{ij}\overline{Y_{ij}}+\overline{Y_{ij}}Y_{ij})  F \nonumber \\
& +\frac{1}{2}\sum_{1\le j\le (N-1)/2} c_{0j}(t) (X_{0j}\overline{X_{0j}}+\overline{X_{0j}}X_{0j})F.\nonumber 
\end{align}
Here we have used the notations:
\begin{align}
X_{ij}&=\overline{u}_i\overline{\partial_j}-u_j\partial_i, & \overline{X}_{ij}= u_i\partial_j-\overline{u}_j\overline{\partial_i},\\ 
Y_{ij} &= \overline{u}_i\partial_j-\overline{u}_j\partial_i, & \overline{Y}_{ij}=u_i\overline{\partial_j}-u_j\overline{\partial_i},\nonumber \\
X_{0j} &= u_0\partial_j-\overline{u}_j\partial_0, & \overline{X_{0j}} = u_0\overline{\partial_j}-u_j\partial_0.\nonumber
\end{align}
\begin{align}
c_{ij}(t) &= \frac{1}{N(\lambda_i(t)-\lambda_j(t))^2}, \\
\tilde{c}_{ij}(t) &=  \frac{1}{N(\lambda_i(t)+\lambda_j(t))^2},\nonumber \\
c_{0j}(t) &= \frac{1}{N\lambda_j^2(t)}. \nonumber
\end{align}
For a fixed $\mathbf{q}\in \mathbf{C}^N$, we let 
\[z_k= \langle \mathbf{q}, v_k\rangle,\quad  k=1,\ldots, (N-1)/2.\] Consider the moments
\begin{gather}
{\mathrm{P}^{(as)}}^{j_0,j_1,\ldots, j_m}_{0 i_1, \ldots, i_m}(z_0,z_1,\ldots,z_{(N-1)/2})=z_0^{2j_0}\prod_{l=1}^m z_{l_i}^{j_i}\bar{z}^{j_i}_{l_i}, \label{eqn: polymoment}\\
{ \mathrm{Q}_{\vec{\lambda}} }^{j_0,j_1,\ldots, j_m}_{0 i_1, \ldots, i_m}(t)=\mathbf{E}^{H_0}({\mathrm{P}^{(as)}}^{j_0,j_1,\ldots, j_m}_{0 i_1, \ldots, i_m}\mid \vec{\lambda})\cdot a(2j_0)^{-1}\prod_{l=1}^m(2^{j_l}j_l!)^{-1}. \label{eqn: asQ}
\end{gather}

The action of the terms $X_{ij}\overline{X_{ij}}$ and $\overline{X_{ij}}X_{ij}$ with $i\neq 0$ on the moments is identical because of the invariance under the permutation $z_i \rightarrow \overline{z}_i$, $i\neq 0$ in the polynomials \eqref{eqn: polymoment}. (This was already used implicitly in \eqref{eqn: dQdt}). Moreover, a calculation shows that $X_{ij}\overline{X_{ij}}$ and $Y_{ij}\overline{Y_{ij}}$ act on the ``anti-symmetric'' moments in identical fashion to $V_{kl}\overline{V}_{kl}$ in the Hermitian case \eqref{eqn: generator1}, \eqref{eqn: generator2}. For $X_{0j}\overline{X_{0j}}$, we have
\begin{align}
X_{0j}\overline{X_{0j}}{ \mathrm{Q}_{\vec{\lambda}} }^{j_1,\ldots, j_m}_{i_1, \ldots, i_m}(t) &= j_k { \mathrm{Q}_{\vec{\lambda}} }^{j_1-1,\ldots,j_k+1,\ldots, j_m}_{i_1, \ldots, i_k, \ldots, i_m}(t) - j_k{ Q_{\vec{\lambda}} }^{j_1,\ldots, j_m}_{i_1, \ldots, i_m}(t), \quad 0\notin \{i_1,\ldots,i_m\}\\
X_{0j}\overline{X_{0j}}{ \mathrm{Q}_{\vec{\lambda}} }^{j_1,\ldots, j_m}_{i_1, \ldots, i_m}(t) &= 2j_0 { \mathrm{Q}_{\vec{\lambda}} }^{j_1,\ldots, j_m}_{i_1, \ldots, i_m}(t) - 2j_0{ \mathrm{Q}_{\vec{\lambda}} }^{j_0-1,\ldots,1,\ldots, j_m}_{0, \ldots, i_k, \ldots, i_m}(t), \quad k\notin \{i_1,\ldots,i_m\}\\
X_{0j}\overline{X_{0j}}{ \mathrm{Q}_{\vec{\lambda}} }^{j_1,\ldots, j_m}_{i_1, \ldots, i_m}(t) &= 2j_0(j_k+1) { \mathrm{Q}_{\vec{\lambda}} }^{j_0-1,\ldots,j_k+1, \ldots, j_m}_{0,i_1, \ldots, i_k, \ldots, i_m}(t)\\
&\quad + (2j_0+1)j_k {\mathrm{Q}_{\vec{\lambda}} }^{j_0+1,\ldots,j_k-1, \ldots, j_m}_{0,i_1, \ldots, i_k, \ldots, i_m}(t) \nonumber \\
&\quad -((2j_0+1)j_k+2j_0(j_k+1)){ \mathrm{Q}_{\vec{\lambda}} }^{j_1,\ldots, j_m}_{i_1, \ldots, i_m}(t). \nonumber
\end{align}

From this we obtain the evolution for the anti-symmetric eigenvector moment flow. We let
\[f^{(as)}_{\vec{\lambda}}(t,\mathbf{\eta})= { \mathrm{Q}_{\vec{\lambda}} }^{j_0,j_1,\ldots, j_m}_{0,i_1, \ldots, i_m}(t)\]
with $\mathbf{\eta}$ defined by the correspondence explained earlier, and $\mathrm{Q}_{\vec{\lambda}}$ is now defined in terms of the anti-symmetric flow as in \eqref{eqn: asQ}. The evolution is given by:
\begin{align}
\partial_t f^{(as)}_{\vec{\lambda}}(t,\mathbf{\eta}) &=\sum_{\substack{j\neq i\\ i,j\ge 1 }} (c_{ij}(t)+\tilde{c}_{ij}(t))\eta_i(1+\eta_j)\left(f(\mathbf{\eta}^{i,j})-f(\mathbf{\eta})\right)\\
&\quad+ \sum_{j\neq 0} c_{0j} 2\eta_0(1+\eta_j)\left(f(\mathbf{\eta}^{0,j})-f(\mathbf{\eta})\right) \nonumber \\
&\quad+ \sum_{i\neq 0} c_{i0} \eta_i(2\eta_0+1)\left(f(\mathbf{\eta}^{i,0})-f(\mathbf{\eta})\right). \nonumber
\end{align}
The maximum principle proof given in \cite[Section 4]{bourgadeyau} to show that $f^{(h)}(t,\mathbf{\eta})$ converges to 1 uniformly in $\mathbf{\eta}$ for fixed particle number $m=\sum \eta_i$, can be reproduced almost exactly for $f^{(as)}_{\vec{\lambda}}(t,\mathbf{\eta})$. We merely need to note that the isotropic local semicircle law proved in \cite[Theorem 2.12]{bloemendaletal}, which is a central ingredient in that proof, applies to general Hermitian matrices, with no non-degeneracy condition on the real part. It thus applies to the anti-symmetric Hermitian matrix $M$.

\end{document}